\documentclass[reqno]{amsart}
\usepackage{amssymb}
\usepackage{mathtools}
\usepackage{a4wide,amsmath}
\usepackage{mathrsfs}
\usepackage{amsthm}
\numberwithin{equation}{section}
\numberwithin{figure}{section}
\numberwithin{table}{section}
\usepackage{bbm}
\usepackage{subfig}
\usepackage{enumerate}
\usepackage{needspace}
\usepackage{graphicx}		  % Tilføjelse af billedfiler
\usepackage{ifpdf}
\ifpdf
\DeclareGraphicsExtensions{.pdf,.eps,.jpg,.png}	
\usepackage[suffix=]{epstopdf}
\fi
\usepackage{xcolor}
\usepackage[utf8]{inputenc}
\usepackage{hyperref}
\hypersetup{hidelinks}

\long\def\MSC#1\EndMSC{\def\arg{#1}\ifx\arg\empty\relax\else
	{\narrower\noindent%
		{2020 Mathematics Subject Classification}: #1\\} \fi}
\long\def\PACS#1\EndPACS{\def\arg{#1}\ifx\arg\empty\relax\else
	{\narrower\noindent%
		{PACS numbers}: #1}\fi}
\long\def\KEY#1\EndKEY{\def\arg{#1}\ifx\arg\empty\relax\else
	{\narrower\noindent% 
		Keywords: #1\\}\fi}

\newcommand{\norm}[1]{\lVert#1\rVert}
\newcommand{\frob}[1]{\norm{#1}_{\textup{F}}}
\newcommand{\abs}[1]{\lvert#1\rvert} 
\newcommand{\inner}[1]{\langle#1\rangle}

\newcommand{\suppm}{\mathop{\textup{supp}}}

\newcommand{\e}{\mathrm{e}}    % Euler's number
\newcommand{\di}{\mathrm{d}}   % differential
\newcommand{\R}{\mathbb{R}}
\newcommand{\N}{\mathbb{N}}

\renewcommand{\H}{\mathcal{H}}
\newcommand{\GammaN}{\Gamma_{\textup{N}}}
\newcommand{\GammaD}{\Gamma_{\textup{D}}}
\newcommand{\DLambda}{\textup{D}\mkern-1.5mu \Lambda}

\newcommand{\Dp}{D_{\textup{F}}^{+}}
\newcommand{\Dm}{D_{\textup{F}}^{-}}

\theoremstyle{plain}
\newtheorem{theorem}{Theorem}[section]
\newtheorem{lemma}[theorem]{Lemma}
\newtheorem{proposition}[theorem]{Proposition}

\theoremstyle{definition}
\newtheorem{definition}[theorem]{Definition}

\newtheorem{assumption}[theorem]{Assumption}
\theoremstyle{remark}
\newtheorem{remark}[theorem]{Remark}

\DeclareMathOperator{\Tr}{tr}

\begin{document}
\title{Direct reconstruction of general elastic inclusions}

\author[S.~Eberle-Blick]{Sarah Eberle-Blick}
\address[S.~Eberle-Blick]{Mathematical Institute for Machine Learning and Data Science, Catholic University of Eichst{\"a}tt-Ingolstadt, Hohe-Schul-Str.~5, 85049 Ingolstadt, Germany.}
\email{Sarah.Eberle-Blick@ku.de}

\author[H.~Garde]{Henrik Garde}
\address[H.~Garde]{Department of Mathematics, Aarhus University, Ny Munkegade~118, 8000 Aarhus C, Denmark.}
\email{garde@math.au.dk}

\author[N.~Hyv\"onen]{Nuutti Hyv\"onen}
\address[N.~Hyv\"onen]{Department of Mathematics and Systems Analysis, Aalto University, P.O. Box~11100, 00076 Helsinki, Finland.}
\email{nuutti.hyvonen@aalto.fi}

\begin{abstract}
	The inverse problem of linear elasticity is to determine the Lam\'e parameters, which characterize the mechanical properties of a domain, from pairs of pressure activations and the resulting displacements on its boundary. This work considers the specific problem of reconstructing inclusions that manifest themselves as deviations from the background Lam\'e parameters. 
	
	The monotonicity method is a direct reconstruction method that has previously been considered for domains only containing positive (or negative) inclusions with finite contrast. That is, all inclusions have previously been assumed to correspond to a finite increase (or decrease) in both Lam\'e parameters compared to their background values. We prove the general outer approach of the monotonicity method that simultaneously allows positive and negative inclusions, of both finite and extreme contrast; the latter refers to either infinitely stiff or perfectly elastic materials.
\end{abstract}

\maketitle

\KEY
linear elasticity,  monotonicity method, perfectly elastic, infinitely stiff.
\EndKEY

\MSC
35R30, 35R05, 47H05.
\EndMSC

\tableofcontents

\section{Introduction}

This work considers the inverse boundary value problem of linear elasticity that can be formulated as the task of determining the Lam\'e parameters, which characterize the mechanical properties of the investigated body, from pairs of external pressure activations and the associated boundary displacements. Previous theoretical works on this inverse problem include~\cite{Barbone04, Beretta14a, Beretta14b, Carstea18, Eskin02, Ikehata90, Ikehata06, Ikehata99, Imanuvilov11,Lin17,Nakamura93,Nakamura94,Nakamura94erratum,Nakamura95}.  

Our interest lies more specifically with the reconstruction of inclusions,~i.e.,~the task of determining from boundary measurements the regions where the Lam\'e parameters differ from their background values. The \emph{monotonicity method}~\cite{Harrach2008,Harrach2013,Tamburrino02} has previously been successfully applied to inclusion detection in the framework of linear elasticity in \cite{Eberle2021,Eberle21b} (see also \cite{Eberle22,Eberle23,Eberle2021a} for related results), but these papers only consider the so-called \emph{inner approach} that determines if a given open set is fully contained in an inclusion. The inner approach assumes that all inclusions are either (i) \emph{positive} or (ii) \emph{negative}, meaning that both Lam\'e parameters are everywhere in the domain either (i) larger or equal or (ii) smaller or equal to their respective background values. Moreover, the cases of \emph{perfectly elastic} and \emph{infinitely stiff} inclusions, which correspond, respectively, to vanishing and infinite Lam\'e parameters, have previously not been considered in connection to the monotonicity method. Such {\em extreme inclusions} can alternatively be characterized by vanishing or infinite Young's modulus, assuming Poisson's ratio is bounded strictly away from $0$ and $\tfrac{1}{2}$ everywhere in the domain; cf.,~e.g.,~\cite{Gould2018}.

Following the analysis on electrical impedance tomography in \cite{Garde2020,Garde2025,Garde2022}, we demonstrate that the previous results on the monotonicity method for linear elasticity can be generalized to the case of extreme inclusions. What is more, we prove that the so-called \emph{outer approach} of the monotonicity method provides simultaneous characterization of positive and negative inclusions, which are allowed to have both finite and extreme contrast; it determines whether a given subset of the domain contains all inclusions. This is given in our main result Theorem~\ref{thm:monoext}. It should be noted, however, that positive (negative) perturbations are located in the same parts of the domain for both Lam\'e parameters, i.e., regions where the Lam\'e parameters deviate to opposite directions from their respective background values are not considered. On the other hand, our result does not require any prior bounds on the contrast of the inclusions compared to the background Lam\'e parameters.

The new results are enabled by proving new \emph{simultaneous} localization results (localized potentials for several pairs of Lam\'e parameters), new monotonicity inequalities that account for extreme inclusions, and operator convergence of Neumann-to-Dirichlet maps for sequences of finite pairs of Lam\'e parameters converging to an extreme pair of Lam\'e parameters.

Section~\ref{sec:forward} introduces the forward problem of linear elasticity with extreme inclusions, and Section~\ref{sec:limit} formulates how it can be reached as a limit of standard forward problems by letting the Lam\'e parameters converge to zero and infinity, respectively, inside the perfectly elastic and infinitely stiff inclusions (cf.~\cite{Garde2020}). Our main result (Theorem~\ref{thm:monoext}) on the outer approach, of reconstructing inhomogeneities composed of both positive and negative inclusions with extreme and non-extreme components, and specialized results for the inner approach (Theorems~\ref{thm:mononeg} and~\ref{thm:mononeg}), of reconstructing definite extreme inclusions, are presented in Section~\ref{sec:mono}. Section~\ref{sec:linear_outer} formulates the linearized outer approach for reconstructing non-extreme inclusions (Theorem~\ref{thm:monofinitelin}), again allowing both positive and negative inclusions, which is a result missing from previous literature on the monotonicity method in the framework of linear elasticity. Sections~\ref{sec:projection} and~\ref{sec:monoineq}--\ref{sec:localized} introduce auxiliary tools and prove lemmas needed for proving the main results in Sections~\ref{sec:convproof} and~\ref{sec:monoextproof}--\ref{sec:monofinitelinptoof}.

\subsection{Remarks on notation} \label{sec:notation}

Due to the application of elasticity, we assume that all involved function spaces are real. 

We denote by $A:B$ the Frobenius inner product of two matrices and by $\frob{ \, \cdot \,}$ the Frobenius norm. We use $\inner{\,\cdot\,,\,\cdot\,}$ for the standard inner product on $L^2(\GammaN)^d$ (where $\GammaN$ is defined in Section~\ref{sec:forward}), with $\norm{\,\cdot\,}$ as the associated norm. 

For a vector-valued function $u$, $\widehat{\nabla}u=\frac{1}{2}\bigl(\nabla u+(\nabla u)^{\textup{T}}\bigr)$ is the symmetric gradient. The divergence of a $d\times d$ matrix-valued function $A$ is
\begin{equation*}
	\nabla \cdot A = \sum_{i,j=1}^d \frac{\partial A_{ij}}{\partial x_j}\e_i,
\end{equation*}
where $\{\e_i\}_{i=1}^d$ is the standard basis for $\R^d$.

We denote by $C^\circ$ the interior of a set $C$. Moreover, $K>0$ denotes a generic positive constant that may change from line to line in the mathematical contents. At all occurrences of using supremum/infimum, we mean the \emph{essential} supremum/infimum.

For self-adjoint operators $A,B\in\mathscr{L}(H)$ on a Hilbert space $H$, we write $A\geq B$ when $A-B$ is positive semidefinite; this is the Loewner ordering of such operators.

\section{Forward problem with extreme inclusions} \label{sec:forward}

We consider a bounded Lipschitz domain, with connected complement, $\Omega$ in $\R^d$ for $d\in\N\setminus\{1\}$, occupied by an isotropic material with linear stress-strain relation. The non-empty open boundary pieces $\GammaD, \GammaN \subset \partial \Omega$ are the corresponding Dirichlet and Neumann boundaries. The choice of mixed boundary conditions is based on the physical treatment of the elasticity problem. The Neumann-to-Dirichlet operator with a fixed Dirichlet part is an idealized model for fixing an elastic object in place on one part of the boundary, applying different pressure patterns to the remaining part, and measuring the resulting displacements.

Let $C_0$ be the part of the domain that is perfectly elastic and $C_\infty$ the part of the domain that is infinitely stiff. We assume the extreme inclusions satisfy the following conditions.
\begin{assumption} \label{assump}
	Let $C = C_0\cup C_\infty$ where $C_0,C_\infty \Subset \Omega$ satisfy:
	\begin{enumerate}[(i)]
		\item $C_0$ and $C_\infty$ are closures of open sets with Lipschitz boundaries.
		\item $C_0 \cap C_\infty = \emptyset$.
		\item $\Omega\setminus C_0$ is connected.
	\end{enumerate}
\end{assumption}
\noindent Note that $C_0$ and $C_\infty$ can consist of several connected components and are also allowed to be empty.
	
If 
\begin{equation*}
	\lambda_0, \mu_0 \in L_+^\infty(\Omega) = \{\, w\in L^\infty(\Omega) : \inf(w)>0 \,\}
\end{equation*}
denote typical non-extreme Lam\'e parameters, we can formally allow extreme Lam\'e parameters $\lambda$ and $\mu$ by writing 
\begin{equation} \label{eq:extremeLame}
    \lambda=\begin{cases}
    \lambda_0 & \text{in } \Omega\setminus C\\
    0 & \text{in } C_0\\
    \infty & \text{in } C_\infty
    \end{cases}
    \quad \text{and}
    \quad
    \mu=\begin{cases}
    \mu_0 & \text{in } \Omega\setminus C\\
    0 & \text{in } C_0\\
    \infty & \text{in } C_\infty.
    \end{cases}
\end{equation}
The displacement vector $u \colon \Omega\setminus C_0 \rightarrow\R^d$ satisfies the following elliptic boundary value problem:
\begin{equation} \label{eq:problem}
	\begin{split}
		\nabla\cdot \bigl(\lambda(\nabla\cdot u)I+2\mu\widehat{\nabla}u\bigr) &=0 \text{ in } \Omega\setminus C, \\
		u &= 0 \text{ on } \GammaD, \\
		\bigl(\lambda(\nabla\cdot u)I+2\mu\widehat{\nabla}u\bigr)n &=
		\begin{cases}
			g & \text{on } \GammaN, \\
			0 & \text{on } \partial C_0, 
		\end{cases} \\
		\widehat{\nabla}u &= 0 \text{ in } C_\infty^\circ, \\
		\int_{\partial C_i} \bigl(\lambda(\nabla\cdot u)I+2\mu\widehat{\nabla}u\bigr)n\,\di S &= 0 \text{ for each component $C_i$ of $C_\infty$.}
	\end{split}
\end{equation}
Here $\partial\Omega = \overline{\GammaD}\cup\overline{\GammaN}$ and $\GammaD\cap\GammaN = \emptyset$. Moreover, $n$ is the outer unit normal vector to $\Omega\setminus C$, $g\in L^2(\GammaN)^d$ the boundary load (pressure field), and $I$ is the $d\times d$ identity matrix. As $C_0$ is perfectly elastic, the pressure field vanishes on its boundary, and as $C_\infty$ is rigid, the displacement field in its interior is constant. Furthermore, the integral of the pressure field over the boundary of each component of $C_\infty$ must vanish at an equilibrium because otherwise these rigid bodies would be moving.

Note that $u$ is only defined in $\Omega\setminus C_0$. There are many different ways to extend $u$ into $C_0$. In this paper, when it is convenient, we may use the following $H^1$-extension $Eu$ instead of $u$.
\begin{definition} \label{def:extension}
	Let $E \colon H^1(\Omega\setminus C_0)^d \to H^1(\Omega)^d$ be defined as
	\begin{equation*}
		Ew = \begin{cases}
			w & \textup{in } \Omega\setminus C_0, \\
			\widetilde{w} & \textup{in } C_0,
		\end{cases}
	\end{equation*}
	where $\widetilde{w}$ is the unique solution in $H^1(C_0^\circ)^d$ to the Dirichlet problem
	\begin{align*} 
		\nabla\cdot \bigl(\lambda_0(\nabla\cdot \widetilde{w})I+2\mu_0\widehat{\nabla}\widetilde{w}\bigr) &= 0 \text{ in } C_0^\circ,  \\
		\widetilde{w} &= w \text{ on } \partial C_0.
	\end{align*}
\end{definition}
In the rest of the paper, we will not explicitly specify which background Lam\'e parameters and which set $C_0$ that is used for the extension operator $E$. It should be implicitly understood when we write $Eu$, that the background Lam\'e parameters and the corresponding $C_0$ set used in the definition of $u$ are also used for the extension $Eu$.

For the sake of comparison, in case $C_0=C_\infty=\emptyset$, we also write down the standard non-extreme PDE problem:
\begin{align*} 
	\nabla\cdot \bigl(\lambda(\nabla\cdot u)I+2\mu\widehat{\nabla}u\bigr) &=0 \text{ in } \Omega, \\
	u &= 0 \text{ on } \GammaD, \\
	\bigl(\lambda(\nabla\cdot u)I+2\mu\widehat{\nabla}u\bigr)n &= g \text{ on } \GammaN.
\end{align*}

\subsection{Variational formulation}

Consider the Hilbert space
\begin{align*}
    \H_{C_0}^{C_\infty} = \bigl\{\, v\in H^1(\Omega\setminus C_0)^d : \widehat{\nabla}v=0 \text{ in } C_\infty^\circ \text{ and } v|_{\GammaD} = 0 \,\bigr\}.
\end{align*}
We may equip $\H_{C_0}^{C_\infty}$ with either of the inner products 
\begin{align}
   \inner{w,v}_{\lambda,\mu} &= \int_{\Omega\setminus C_0} \lambda_0(\nabla\cdot w)(\nabla\cdot v) + 2\mu_0\widehat{\nabla}w:\widehat{\nabla}v \, \di x, \label{eq:innerweak} \\
   \inner{\widehat{\nabla}w,\widehat{\nabla}v}_{L^2(\Omega\setminus C_0)^{d\times d}} &= \int_{\Omega\setminus C_0}\widehat{\nabla}w:\widehat{\nabla}v \, \di x. \label{eq:innerstar}
\end{align}
Due to a Poincar\'e inequality, based on the vanishing trace on $\GammaD$, and by Korn's inequality, we have (c.f.~\cite[Theorem~2.3]{Cialet2010})
\begin{equation}
	\norm{\nabla v}_{L^2(\Omega\setminus C_0)^{d\times d}} \leq K\norm{\widehat{\nabla} v}_{L^2(\Omega\setminus C_0)^{d\times d}}, \quad v\in\H_{C_0}^{C_\infty}.
\end{equation}
Hence, again by the Poincar\'e inequality, and as 
\begin{equation} \label{eq:nablabnd}
	\abs{\nabla\cdot v} = \abs{\Tr(\widehat{\nabla}v)} \leq \sqrt{d}\frob{\widehat{\nabla}v}, 
\end{equation} 
the norms $\norm{\,\cdot\,}_{\lambda,\mu}$ and $\norm{\widehat{\nabla}\,\cdot\,}_{L^2(\Omega\setminus C_0)^{d\times d}}$ are equivalent to the usual $H^1(\Omega\setminus C_0)^d$-norm on $\H_{C_0}^{C_\infty}$. 
\begin{remark}
  As a slight abuse of notation,
  %to avoid unnecessarily introducing too much redundant notation,
  we occasionally denote by $\inner{\,\cdot\,,\,\cdot\,}_{\lambda,\mu}$ the inner product on  $\H_{C_0}^{\emptyset}$ defined by the same expression \eqref{eq:innerweak}; see, e.g.\ Section~\ref{sec:projection}.
\end{remark}
The weak problem formulation of \eqref{eq:problem} is
\begin{equation} \label{eq:weakform}
	\inner{u,v}_{\lambda,\mu} = \inner{g,v} \quad \text{for all } v \in\H_{C_0}^{C_\infty}.
\end{equation}
As mentioned in Section~\ref{sec:notation}, $\inner{\,\cdot\,,\,\cdot\,}$ is the standard inner product on $L^2(\GammaN)^d$, with $\norm{\,\cdot\,}$ denoting the associated norm. The Lax--Milgram lemma guarantees the existence of a unique solution $u = u_g^{\lambda,\mu}\in \H_{C_0}^{C_\infty}$ to the weak problem \eqref{eq:weakform}; we will often use this notation to make it clear which Neumann data and which Lam\'e parameters are in use. This also gives the continuous dependence on the Neumann data:
\begin{equation*}
	\norm{u_g^{\lambda,\mu}}_{H^1(\Omega\setminus C_0)^d} \leq K\norm{g}.
\end{equation*}

We define the Neumann-to-Dirichlet (ND) map $\Lambda(\lambda,\mu)$ by
\begin{equation*}
	\Lambda(\lambda,\mu)\colon L^2(\GammaN)^d\rightarrow L^2(\GammaN)^d, \quad g\mapsto u_g^{\lambda,\mu}|_{\GammaN}, 
\end{equation*}
which is compact and self-adjoint. From the weak form, we get
\begin{equation}
	\inner{\Lambda(\lambda_1,\mu_1)g,g} = \inner{u_g^{\lambda_2,\mu_2},Eu_g^{\lambda_1,\mu_1}}_{\lambda_2,\mu_2}
\end{equation}
for any pairs $(\lambda_1,\mu_1)$ and $(\lambda_2,\mu_2)$ for which the $C_\infty$ inclusions of the extreme Lam\'e parameter pair  $(\lambda_2,\mu_2)$ are contained in those of $(\lambda_1,\mu_1)$. In particular, if $(\lambda_1,\mu_1) = (\lambda_2,\mu_2)$, this becomes
\begin{equation}
	\inner{\Lambda(\lambda_1,\mu_1)g,g} = \norm{u_g^{\lambda_1,\mu_1}}_{\lambda_1,\mu_1}^2.
\end{equation}

\section{Extreme measurements as a limit} \label{sec:limit}

In this section, we retrieve the displacement vector field and the ND map in the case of perfectly elastic and infinitely stiff inclusions via a limit of problems that correspond to a sequence of truncated Lam\'e parameters that belong to $L^\infty_+(\Omega)$. This also ensures that the PDE formulation \eqref{eq:problem} and its weak form \eqref{eq:weakform} indeed are the correct interpretations of how to model such extreme inclusions.

\begin{theorem}\label{thm:convergence}
Let $C=C_0\cup C_\infty$ satisfy Assumption \ref{assump}, let $\lambda$ and $\mu$ be as in \eqref{eq:extremeLame}, and define the $\epsilon$-truncated versions of $\lambda$ and $\mu$, with $\epsilon>0$, by
\begin{equation*}
    \lambda_\epsilon=\begin{cases}
    \lambda_0 & \text{in } \Omega\setminus C\\
    \epsilon\lambda_0 & \text{in } C_0\\
    \epsilon^{-1}\lambda_0 & \text{in } C_\infty
    \end{cases}
    \quad \text{and} \quad
    \mu_\epsilon=\begin{cases}
    \mu_0 & \text{in } \Omega\setminus C\\
    \epsilon\mu_0 & \text{in } C_0\\
    \epsilon^{-1}\mu_0 & \text{in } C_\infty.
    \end{cases}
\end{equation*}
Then
\begin{equation*} 
	\norm{ Eu_g^{\lambda,\mu} - u_g^{\lambda_\epsilon,\mu_\epsilon}}_{H^1(\Omega)^d} \leq K \epsilon^{1/2} \norm{g}. 
\end{equation*}
As a direct consequence,
\begin{equation*} 
	\norm{\Lambda(\lambda,\mu) - \Lambda(\lambda_\epsilon,\mu_\epsilon)}_{\mathscr{L}(L^2(\GammaN)^d)} \leq K\epsilon^{1/2}. 
\end{equation*}
\end{theorem}
\begin{proof}
	The proof is given in Section~\ref{sec:convproof}.
\end{proof}

\section{Reconstruction of extreme inclusions} \label{sec:mono}

Next, we will consider the reconstruction of inclusions via the monotonicity method. In order to obtain exact reconstruction methods, a unique continuation principle is needed for the background parameters; see Section~\ref{sec:localized}. In particular, it ensures the existence of certain localized solutions used for proving the more challenging ``only if'' direction of our reconstruction methods, by providing more local control over monotonicity inequalities for the ND maps.
\begin{definition} \label{def:ucp}
	Suppose $U\subseteq \overline{\Omega}$ is a relatively open and connected set that intersects $\GammaN$. We say that $\lambda_0$ and $\mu_0$ satisfy the weak \emph{unique continuation principle} (UCP) in $U$ for the linear elasticity problem, provided that only the trivial solution of 
	\begin{equation*}
		\nabla\cdot \bigl(\lambda_0(\nabla\cdot v)I+2\mu_0\widehat{\nabla}v\bigr) = 0 \text{ in } U^\circ
	\end{equation*}
	can be identically zero in a non-empty open subset of $U$, and likewise, only the trivial solution has vanishing Cauchy data on $\partial U\cap \GammaN$. If $U = \overline{\Omega}$ we simply say that $\lambda_0$ and $\mu_0$ satisfy the UCP.
\end{definition}
\begin{remark} \label{remark:ucp}
	From \cite[Theorem~1.2]{Lin2011} and \cite[Corollary~2.2]{Eberle2021a} (and the subsequent remarks therein), it follows that a sufficient condition for the UCP to hold for $\lambda_0$ and $\mu_0$ is that $\Omega$ is a bounded $C^{1,1}$ domain, $\lambda_0\in L^\infty_+(\Omega)$, and $\mu_0\in L^\infty_+(\Omega)\cap C^{0,1}(\Omega)$.
\end{remark}
We introduce a family of admissible test inclusions as
\begin{align*}
	\mathcal{A} &= \{\, C \Subset \Omega : C \text{ is the closure of an open set,}  \\
	&\hphantom{{}= \{C \Subset \Omega :{}\,}\text{has connected complement,} \\
	&\hphantom{{}= \{C \Subset \Omega :{}\,}\text{and has Lipschitz boundary } \partial C \,\}.
\end{align*}
\begin{definition} \label{def:posneginc}
	Consider mutually disjoint (possibly empty) sets $\Dp$, $\Dm$, $D_0$ and $D_\infty$, with $D_0,D_\infty\in\mathcal{A}$ and $\Dp,\Dm \Subset \Omega$ measurable. Let $\lambda_0,\mu_0,\lambda_{\pm},\mu_{\pm}\in L^\infty_+(\Omega)$ be such that
	\begin{equation*}
		\lambda_-\leq\lambda_0\leq\lambda_+ \quad\text{and}\quad \mu_-\leq\mu_0\leq\mu_+.
	\end{equation*} 
	We say that $\lambda$ and $\mu$ have \emph{inclusions} in $D = \Dp\cup\Dm\cup D_0\cup D_\infty$ if they are of the form
	\begin{equation} \label{eq:posinclame}
		\lambda = \begin{cases}
			\lambda_0 & \text{in } \Omega\setminus D \\
			\lambda_- & \text{in } \Dm \\
			\lambda_+ & \text{in } \Dp \\
			0 & \text{in } D_0 \\
			\infty & \text{in } D_\infty
		\end{cases}
		\quad\text{and}\quad
		\mu = \begin{cases}
			\mu_0 & \text{in } \Omega\setminus D \\
			\mu_- & \text{in } \Dm \\
			\mu_+ & \text{in } \Dp \\
			0 & \text{in } D_0 \\
			\infty & \text{in } D_\infty.
		\end{cases}
	\end{equation}
\end{definition}
Note that $\lambda_{\pm}$ and $\mu_{\pm}$ are not required to deviate from the background parameters at the same locations inside $D_{\textup{F}}^{\pm}$. 

The following technical assumption ensures that certain pathological cases cannot occur. It allows $\mu_{\pm}$ to e.g.\ jump or exhibit a strict local increase/decrease away from $\mu_0$ near $\partial D$. At a positive distance from $\partial D$, there are no such restrictions on $\mu_{\pm}$, and likewise, there are no such restrictions on $\lambda_{\pm}$ in general (even near $\partial D$). 
\begin{assumption} \label{assump:technical}
	Consider the setting in Definition~\ref{def:posneginc}. For every $x\in\partial D$ and every open neighborhood $W$ of $x$, there exists a relatively open connected set $V\subset D$ that intersects $\partial D$, and $V\subset \widetilde{D} \cap W$ for a set $\widetilde{D}\in \{\Dp,\Dm,D_0, D_\infty\}$. Moreover,
	\begin{itemize}
		\item If $\widetilde{D} = \Dp$, there exists an open ball $B\subset V$ such that $\inf_B(\mu_+ -\mu_0)>0$.
		\item If $\widetilde{D} = \Dm$, there exists an open ball $B\subset V$ such that $\sup_B(\mu_- -\mu_0)<0$.
	\end{itemize}
\end{assumption}
Our main result is Theorem~\ref{thm:monoext}, which assumes \emph{no information} on the magnitude of the perturbations and allows both positive and negative perturbations. It corresponds to building the reconstruction of the inclusions from larger sets, i.e.\ checking if a test inclusion $C$ contains $D$. For the test operators in the reconstruction method, we use the notation:
\begin{itemize}
	\setlength\itemsep{.2em}
	\item $\Lambda_{C}^{\emptyset}$ is the ND map with Lam\'e parameters $0$ in $C$, and with $\lambda_0$ and $\mu_0$ elsewhere.
	\item $\Lambda_{\emptyset}^C$ is the ND map with Lam\'e parameters $\infty$ in $C$, and with $\lambda_0$ and $\mu_0$ elsewhere.
\end{itemize} 
\begin{theorem} \label{thm:monoext} \needspace{2\baselineskip} {}\
	Let $\lambda$ and $\mu$ have inclusions in $D\in\mathcal{A}$, satisfying Assumption~\ref{assump:technical}. For $C\in\mathcal{A}$,
	\begin{equation*}
		D\subseteq C \quad \text{implies} \quad \Lambda_{C}^{\emptyset} \geq \Lambda(\lambda,\mu) \geq \Lambda_{\emptyset}^{C}.
	\end{equation*}
	Moreover, if $\lambda_0$ and $\mu_0$ satisfy the UCP, then
	\begin{equation*}
		\Lambda_{C}^{\emptyset} \geq \Lambda(\lambda,\mu) \geq \Lambda_{\emptyset}^{C} \quad \text{implies} \quad D\subseteq C.
	\end{equation*}
\end{theorem}
\begin{proof}
	The proof is given in Section~\ref{sec:monoextproof}.
\end{proof}
In case the UCP holds, Theorem~\ref{thm:monoext} yields
\begin{equation*}
	D = \cap\,\{\, C\in\mathcal{A} : \Lambda_{C}^{\emptyset} \geq \Lambda(\lambda,\mu) \geq \Lambda_{\emptyset}^{C} \,\}.
\end{equation*}
Moreover, the proof of the theorem allows the following conclusions: If there are only positive inclusions near $\partial D$ (from $\Dp$ or $D_\infty$), then one only needs to consider the inequality $\Lambda(\lambda,\mu) \geq \Lambda_{\emptyset}^{C}$. Likewise, if there are only negative inclusions near $\partial D$ (from $\Dm$ or $D_0$), then only the inequality $\Lambda_{C}^{\emptyset} \geq \Lambda(\lambda,\mu)$ needs to be considered.

The following two theorems correspond to special cases that allow building the reconstruction of the inclusions from the inside, i.e.\ checking if an open set $B$ is contained in $D$. For simplicity, we only present these results for extreme inclusions. In the presence of non-extreme inclusions, the bounds on $\beta$ would depend on \emph{a priori} lower bounds on the contrast of the finite perturbations.

In Theorems~\ref{thm:monopos} and \ref{thm:mononeg}, we use the following short-hand notations:
\begin{align*}
	\Lambda_0 &= \Lambda(\lambda_0,\mu_0), \\
	\Lambda_{\beta,B} &= \Lambda(\lambda_0+\beta\chi_B,\mu_0+\beta\chi_B), \\
	\inner{\DLambda_{\beta,B}g,g} &= -\beta\int_B \abs{\nabla\cdot u_g^{\lambda_0,\mu_0}}^2 + 2\frob{\widehat{\nabla}u_g^{\lambda_0,\mu_0}}^2\,\di x.
\end{align*}
Note that $\DLambda_{\beta,B}$ is the Frech\'et derivative of $\Lambda$ at $(\lambda_0,\mu_0)$ in the direction $(\beta\chi_B,\beta\chi_B)$, where $\chi_B$ is the characteristic function of a measurable set $B\subseteq\Omega$; see \cite[Lemma~2.3]{Eberle2021}.

\begin{theorem} \label{thm:monopos}
	Let $\beta>0$ for $\Lambda_{\beta,B}$ and $0<\beta < \min\{\inf(\lambda_0),\inf(\mu_0)\}$ for $\DLambda_{\beta,B}$. Let $B\subseteq\Omega$ be an open set and let $D\in\mathcal{A}$. Then
	\begin{equation*}
		B \subset D \quad \text{implies} \quad \Lambda_{\beta,B} \geq \Lambda_{\emptyset}^{D} \quad\text{and}\quad \Lambda_0+\DLambda_{\beta,B} \geq \Lambda_{\emptyset}^{D}.
	\end{equation*}
	Moreover, if $\lambda_0$ and $\mu_0$ satisfy the UCP, then
	\begin{equation*}
	  \text{both} \quad \Lambda_{\beta,B} \geq \Lambda_{\emptyset}^{D} \quad\text{and}\quad \Lambda_0+\DLambda_{\beta,B} \geq \Lambda_{\emptyset}^{D} \quad \text{imply} \quad B\subset D.
	\end{equation*}
\end{theorem}
\begin{proof}
	The proof is given in Section~\ref{sec:monoposproof}.
\end{proof}

\begin{theorem} \label{thm:mononeg}
	Let $0<\beta < \min\{\inf(\lambda_0),\inf(\mu_0)\}$, let $B\subseteq\Omega$ be an open set, and let $D\in\mathcal{A}$. Then
	\begin{equation*}
		B \subset D \quad \text{implies} \quad \Lambda_{D}^{\emptyset} \geq \Lambda_{-\beta,B} \quad\text{and}\quad \Lambda_{D}^{\emptyset} \geq \Lambda_0+\DLambda_{-\beta,B}.
	\end{equation*}
	Moreover, if $\lambda_0$ and $\mu_0$ satisfy the UCP, then
	\begin{equation*}
	  \text{both}\quad \Lambda_{D}^{\emptyset} \geq \Lambda_{-\beta,B} \quad\text{and}\quad \Lambda_{D}^{\emptyset} \geq \Lambda_0+\DLambda_{-\beta,B} \quad \text{imply} \quad B\subset D.
	\end{equation*}
\end{theorem}
\begin{proof}
	The proof is given in Section~\ref{sec:mononegproof}.
\end{proof}
Theorems~\ref{thm:monopos} and \ref{thm:mononeg} allow the representations
\begin{align*}
	D^\circ &= \cup\,\{\, B\subset \Omega \text{ open ball} : \Lambda_{\beta,B} \geq \Lambda_{\emptyset}^{D} \,\} \\
	&= \cup\,\{\, B\subset \Omega \text{ open ball} : \Lambda_0+\DLambda_{\beta,B} \geq \Lambda_{\emptyset}^{D} \,\} \\
	&= \cup\,\{\, B\subset \Omega \text{ open ball} : \Lambda_{D}^{\emptyset} \geq \Lambda_{-\beta,B} \,\} \\
	&= \cup\,\{\, B\subset \Omega \text{ open ball} : \Lambda_{D}^{\emptyset} \geq \Lambda_0+\DLambda_{-\beta,B} \,\},
\end{align*}
assuming the UCP holds and $\beta$ is chosen accordingly. Of course, ``open ball'' may be replaced by some other basis of open sets for the Euclidean topology.

\section{Linearized outer reconstruction of non-extreme inclusions} \label{sec:linear_outer}

Regarding Definition~\ref{def:posneginc}, we call the inclusions \emph{non-extreme} if $D_0=\emptyset=D_\infty$, such that $D = \Dm\cup\Dp$. Since the outer approach from Theorem~\ref{thm:monoext} has not been explored earlier for linear elasticity with non-extreme inclusions, we investigate this topic a bit further here. See also \cite{Garde2025,Garde2019,Harrach2013} for related considerations in the context of electrical impedance tomography. 

In the non-extreme case, as a corollary to Theorem~\ref{thm:monoext} and Lemma~\ref{lemma:monoback}, one could use a finite perturbation in the test operators based on the magnitudes of $\abs{\lambda_{\pm}-\lambda_0}$ and $\abs{\mu_{\pm}-\mu_0}$. Since this is such a straightforward consequence of the extreme result, we do not elaborate further.

The more important result in the case of non-extreme inclusions is that the outer approach can be formulated with linearized test operators, which is useful for speeding up a numerical implementation. The downside, compared to Theorem~\ref{thm:monoext}, is that we now need to know bounds on the magnitudes of the perturbations. However, unlike the extreme case, we do not require Lipschitz boundary regularity of the admissible test inclusions here: 
\begin{equation*}
	\widehat{\mathcal{A}} = \{\, C \Subset \Omega : C \text{ is the closure of an open set and has connected complement} \,\}.
\end{equation*}

The Frech\'et derivative of $\Lambda$ at $(\lambda_0,\mu_0)$ in the direction $(\eta_\lambda,\eta_\mu)$ is given by (see~\cite[Lemma~2.3]{Eberle2021})
\begin{equation*}
	\inner{\DLambda(\lambda_0,\mu_0;\eta_\lambda,\eta_\mu)g,g} = -\int_\Omega \eta_\lambda\abs{\nabla\cdot u_g^{(\lambda_0,\mu_0)}}^2 + 2\eta_\mu\frob{\widehat{\nabla}u_g^{(\lambda_0,\mu_0)}}^2\,\di x.
\end{equation*}
We assume there are \emph{known} positive scalars satisfying
\begin{alignat*}{2}
	\alpha_\lambda &\leq \min\{\inf(\lambda_0),\inf(\lambda)\}, &\qquad \alpha_\mu &\leq \min\{\inf(\mu_0),\inf(\mu)\}, \\
	\beta_\lambda &\geq \max\{\sup(\lambda_0),\sup(\lambda)\}, &\qquad \beta_\mu &\geq \max\{\sup(\mu_0),\sup(\mu)\},
\end{alignat*}
and we define the test operators
\begin{align*}
	\DLambda_C^+ &= \DLambda(\lambda_0,\mu_0;(\beta_\lambda-\lambda_0)\chi_C,(\beta_\mu-\mu_0)\chi_C), \\
	\DLambda_C^- &= \DLambda(\lambda_0,\mu_0;(\lambda_0-\tfrac{\beta_\lambda^2}{\alpha_\lambda})\chi_C,(\mu_0-\tfrac{\beta_\mu^2}{\alpha_\mu})\chi_C),
\end{align*}
where $\chi_C$ is a characteristic function on $C\in\widehat{\mathcal{A}}$.

\begin{theorem} \label{thm:monofinitelin}  {}\
	Let $\lambda$ and $\mu$ have \emph{non-extreme} inclusions in $D\in\widehat{\mathcal{A}}$, satisfying Assumption~\ref{assump:technical}. For $C\in\widehat{\mathcal{A}}$,
	\begin{equation*}
		D\subseteq C \quad \text{implies} \quad \DLambda_C^- \geq \Lambda(\lambda,\mu)-\Lambda(\lambda_0,\mu_0) \geq \DLambda_C^+.
	\end{equation*}
	Moreover, if $\lambda_0$ and $\mu_0$ satisfy the UCP, then
	\begin{equation*}
		\DLambda_C^- \geq \Lambda(\lambda,\mu)-\Lambda(\lambda_0,\mu_0) \geq \DLambda_C^+ \quad \text{implies} \quad D\subseteq C.
	\end{equation*}
\end{theorem}
\begin{proof}
	The proof is given in Section~\ref{sec:monofinitelinptoof}.
\end{proof}
Theorem~\ref{thm:monofinitelin} yields
\begin{equation*}
	D = \cap\,\{\, C\in\mathcal{A} : \DLambda_C^- \geq \Lambda(\lambda,\mu)-\Lambda(\lambda_0,\mu_0) \geq \DLambda_C^+ \,\},
\end{equation*}
assuming the UCP holds. Moreover, similar to the discussion after Theorem~\ref{thm:monoext}, if there are only positive non-extreme inclusions near $\partial D$, then we only need to consider the inequality $\Lambda(\lambda,\mu)-\Lambda(\lambda_0,\mu_0) \geq \DLambda_C^+$. If there are only negative non-extreme inclusions near $\partial D$, then we only need to consider the inequality $\DLambda_C^- \geq \Lambda(\lambda,\mu)-\Lambda(\lambda_0,\mu_0)$.

\section{Orthogonal projections} \label{sec:projection}

In this section, we assume $C_0$ and $C_\infty$, with $C = C_0\cup C_\infty$, satisfy Assumption~\ref{assump}, and let $\lambda$ and $\mu$ be as in \eqref{eq:extremeLame}, with background parameters $\lambda_0,\mu_0\in L^\infty_+(\Omega)$ and extreme inclusions in $C$.

We define two orthogonal projections:
\begin{itemize}
	\setlength\itemsep{.2em}
	\item $P$ is the orthogonal projection of $\H_{C_0}^{\emptyset}$ onto $\H_{C_0}^{C_\infty}$ w.r.t.\ $\inner{\,\cdot\,,\,\cdot\,}_{\lambda,\mu}$.
	\item $P^\perp$ is the orthogonal projection of $\H_{C_0}^{\emptyset}$ onto $(\H_{C_0}^{C_\infty})^\perp$ w.r.t.\ $\inner{\,\cdot\,,\,\cdot\,}_{\lambda,\mu}$.
\end{itemize}
\begin{lemma} \label{lemma:altcharacterization}
	For any $g\in L^2(\GammaN)^d$, $u_g^{\lambda,\mu}$ can be uniquely characterized as follows: 
	\begin{equation*}
		u_g^{\lambda,\mu} = P(u_g^{\lambda_0,\mu_0}|_{\Omega\setminus C_0} - w),
	\end{equation*}
	where $w$ is the unique solution in $H^1(\Omega\setminus C_0)^d$ to
	\begin{align*} 
		\nabla\cdot \bigl(\lambda_0(\nabla\cdot w)I+2\mu_0\widehat{\nabla}w\bigr) &=0 \text{ in } \Omega\setminus C_0, \\
		w &= 0 \text{ on } \GammaD, \\
		\bigl(\lambda_0(\nabla\cdot w)I+2\mu_0\widehat{\nabla}w\bigr)n &= \begin{cases}
			0 &\text{on } \GammaN, \\
			\bigl(\lambda_0(\nabla\cdot u_g^{\lambda_0,\mu_0})I+2\mu_0\widehat{\nabla}u_g^{\lambda_0,\mu_0}\bigr)n &\text{on } \partial C_0.
		\end{cases}
	\end{align*}
\end{lemma}
\begin{proof}
	From the boundary value problems defining $u_g^{\lambda_0,\mu_0}$ and $w$, it is clear that
	\begin{equation}
		\widehat{u} = u_g^{\lambda_0,\mu_0}|_{\Omega\setminus C_0} - w
	\end{equation}
	corresponds to the case with the extreme inclusion $C_0$, but without $C_\infty$. Recall that $P$ is an orthogonal projection and hence self-adjoint with respect to the inner product \eqref{eq:innerweak}. Due to the weak formulation satisfied by  $\widehat{u}$, the following thus holds for all $Pv = v\in \H_{C_0}^{C_\infty}\subseteq\H_{C_0}^{\emptyset}$:
	\begin{equation*}
		\inner{g,v} = \inner{\widehat{u},Pv}_{\lambda,\mu} = \inner{P\widehat{u},v}_{\lambda,\mu}.
	\end{equation*}
	By the unique solvability of the variational problem \eqref{eq:weakform}, we conclude that $u_g^{\lambda,\mu} = P\widehat{u}$.
\end{proof}

\begin{lemma}\label{lemma:perp}
	If $v \in \H_{C_0}^{\emptyset}$ satisfies $\nabla\cdot \bigl( \lambda(\nabla\cdot v)I+2\mu\widehat{\nabla}v \bigr)=0$ in $\Omega\setminus C$, then there is a constant $K>0$ (independent of $v$) such that
	\begin{equation*}
		\norm{P^\perp v}_{\lambda,\mu} \leq K\norm{\widehat{\nabla}v}_{L^2(C_\infty)^{d\times d}}.
	\end{equation*}    
\end{lemma}
\begin{proof}
	Let $v^{C_\infty}$ denote the piecewise constant function in $H^1(C_\infty^\circ)^d$ that in each component $C_i$ of $C_\infty$ equals the mean value of the restriction $v|_{C_i}$. Let $w$ be the unique solution in $\H_{C_0}^{C_\infty}$ of the auxiliary problem
	\begin{equation} 
		\begin{split}
			\nabla\cdot \bigl(\lambda(\nabla\cdot w)I+2\mu\widehat{\nabla}w\bigr) &=0 \text{ in } \Omega\setminus C, \\
			w &= v \text{ on } \partial(\Omega\setminus C_0), \\
			w &= v^{C_\infty} \text{ in } C_\infty^\circ.
		\end{split}
	\end{equation}
	Clearly, $v-w$ satisfies a corresponding boundary value problem in $\Omega\setminus C$, with a Dirichlet condition that is only non-trivial on $\partial C_\infty$. By continuous dependence on the Dirichlet data, the trace theorem in $C_\infty^\circ$, and applying  Poincare's and Korn's inequalities on each component $C_i$, we obtain the bound
	\begin{align*}
		\norm{v-w}_{H^1(\Omega\setminus C_0)^d}^2 &= \norm{v-w}_{H^1(\Omega\setminus C)^d}^2 + \norm{v-v^{C_\infty}}_{H^1(C_\infty^\circ)^d}^2 \\
		&\leq K\bigl(\norm{v-v^{C_\infty}}_{H^{1/2}(\partial C_\infty)^d}^2 + \norm{v-v^{C_\infty}}_{H^1(C_\infty^\circ)^d}^2 \bigr) \\
		&\leq K\norm{v-v^{C_\infty}}_{H^1(C_\infty^\circ)^d}^2 \\
		&\leq K\norm{\widehat{\nabla}v}_{L^2(C_\infty)^{d\times d}}^2.
	\end{align*}
	Finally, the projection theorem and norm equivalences give
	\begin{align*}
		\norm{P^\perp v}_{\lambda,\mu} &= \inf_{\widehat{v}\in\H_{C_0}^{C_\infty}}\norm{v-\widehat{v}}_{\lambda,\mu} \leq \norm{v-w}_{\lambda,\mu} \\
		&\leq K\norm{v-w}_{H^1(\Omega\setminus C_0)^d} \leq K \norm{\widehat{\nabla}v}_{L^2(C_\infty)^{d\times d}}. \qedhere
	\end{align*}
\end{proof}

\section{Proof of Theorem~\ref{thm:convergence}} \label{sec:convproof}

Let $u_\epsilon = u_g^{\lambda_\epsilon,\mu_\epsilon} \in \H_{\emptyset}^{\emptyset}$ and $u = u_g^{\lambda,\mu} \in \H_{C_0}^{C_\infty}$ be the solutions to \eqref{eq:weakform} with the corresponding Lam\'e parameters. To more easily distinguish between the weak formulations of $u_\epsilon$ and $u$, we write down \eqref{eq:weakform} specifically for $u_\epsilon$:
\begin{equation}
	\int_\Omega \lambda_\epsilon(\nabla\cdot u_\epsilon)(\nabla\cdot v) + 2\mu_\epsilon\widehat{\nabla}u_\epsilon:\widehat{\nabla}v \, \di x = \inner{g,v} \quad \text{for all } v \in\H_{\emptyset}^{\emptyset}.
	\label{eq:weakformeps}
\end{equation}

First, we give some estimates for  $u_\epsilon$ in terms of $g$. Using \eqref{eq:weakformeps}, the ``Cauchy inequality with $\varepsilon$" (with $\alpha >0$ taking the role of $\varepsilon$), and the trace theorem in $\Omega\setminus C$, leads to
\begin{align*}
	2\inf(\mu_0)\bigl( \norm{\widehat{\nabla}u_\epsilon}_{L^2(\Omega\setminus C)^{d\times d}}^2 + \epsilon\norm{\widehat{\nabla}u_\epsilon}_{L^2(C_0)^{d\times d}}^2 + \epsilon^{-1}\norm{\widehat{\nabla}u_\epsilon}_{L^2(C_\infty)^{d\times d}}^2 \bigr) \hspace{-5cm}&\\
	&\leq
	\int_{\Omega} \lambda_\epsilon \abs{\nabla\cdot u_\epsilon}^2 + 2\mu_\epsilon \frob{\widehat{\nabla}u_\epsilon}^2 \,\di x \\
	&= \inner{g,u_\epsilon} \\ 
	&\leq \alpha\norm{g}^2 + \frac{1}{4\alpha}\norm{u_\epsilon}^2_{H^{1/2}(\partial(\Omega\setminus C))^d} \\
	&\leq \alpha\norm{g}^2 + \frac{\widehat{K}}{4\alpha} \norm{\widehat{\nabla}u_\epsilon}_{L^2(\Omega\setminus C)^{d\times d}}^2.
\end{align*}
Choosing $\alpha=\frac{\widehat{K}}{8\inf(\mu_0)}$ and restructuring results in 
\begin{align*}
	\epsilon\norm{\widehat{\nabla}u_\epsilon}_{L^2(C_0)^{d\times d}}^2 + \epsilon^{-1}\norm{\widehat{\nabla}u_\epsilon}_{L^2(C_\infty)^{d\times d}}^2 \leq \frac{\widehat{K}}{16\inf(\mu_0)^2}\norm{g}^2.
\end{align*}
Bounding the left-hand side from below, and absorbing all constants in $K$, leads to the following estimates on $C_\infty$ and $C_0$:
\begin{align}
	\norm{\widehat{\nabla}u_\epsilon}_{L^2(C_\infty)^{d\times d}} &\leq K\epsilon^{1/2}\norm{g}, \label{eq:convbnd1}\\
	\norm{\widehat{\nabla}u_\epsilon}_{L^2(C_0)^{d\times d}} &\leq K\epsilon^{-1/2}\norm{g}. \label{eq:convbnd2}
\end{align}
We will return to these bounds later.

We proceed to prove the first part of the theorem. In the following computations we consider $u_\epsilon|_{\Omega\setminus C_0}\in\H_{C_0}^{\emptyset}$. By the Cauchy--Schwarz inequality,
\begin{align}
	\norm{u-u_\epsilon}_{\lambda,\mu}^2 &=  \norm{P(u-u_\epsilon)}_{\lambda,\mu}^2 + \norm{P^\perp(u-u_\epsilon)}_{\lambda,\mu}^2 \notag \\  
	&\leq \inner{u-u_\epsilon,P(u-u_\epsilon)}_{\lambda,\mu}  + \norm{P^\perp u_\epsilon}_{\lambda,\mu} \norm{u-u_\epsilon}_{\lambda,\mu}. \label{eq:est_u} 
\end{align} 
As Lemma~\ref{lemma:perp} and \eqref{eq:convbnd1} yield
\begin{equation} \label{eq:firsttermbnd}
	\norm{P^\perp u_\epsilon}_{\lambda,\mu}\norm{u-u_\epsilon}_{\lambda,\mu} \leq K\epsilon^{1/2}\norm{g}\norm{u-u_\epsilon}_{\lambda,\mu},
\end{equation}
we just need a similar type of bound for the first term on the right-hand side of \eqref{eq:est_u}.

From the weak formulations \eqref{eq:weakform} and \eqref{eq:weakformeps} for $u$ and $u_\epsilon$, respectively, we get
\begin{align}
	\inner{u-u_\epsilon,P(u-u_\epsilon)}_{\lambda,\mu} &= \inner{u,P(u-u_\epsilon)}_{\lambda,\mu} - \inner{u_\epsilon,P(u-u_\epsilon)}_{\lambda,\mu} \notag \\
	&= \epsilon \int_{C_0} \lambda_0(\nabla\cdot u_\epsilon)(\nabla\cdot w) + 2\mu_0 \widehat{\nabla} u_\epsilon : \widehat{\nabla} w \,\di x, 
	\label{eq:firstterm}
\end{align} 
where $w$ is an $H^1$ extension of $P(u-u_\epsilon)$ onto $C_0$. This only gives a contribution on $C_0$, as $\widehat{\nabla}P(u-u_\epsilon) = 0$ in $C^\circ_\infty$ and therefore also $\nabla\cdot P(u-u_\epsilon) = 0$ in $C^\circ_\infty$ by \eqref{eq:nablabnd}.

We adopt the notation that ``int'' refers to taking a trace from within $C_0^\circ$.
%, to respect the jump relation for the normal derivative of $u_\epsilon$ at $\partial C_0$.
Since $u_\epsilon$ satisfies
\begin{equation} \label{eq:pdeinC0}
	\nabla\cdot \bigl(\lambda_0(\nabla\cdot u_\epsilon)I+2\mu_0\widehat{\nabla}u_\epsilon\bigr) = \nabla\cdot \bigl(\lambda_\epsilon(\nabla\cdot u_\epsilon)I+2\mu_\epsilon\widehat{\nabla}u_\epsilon\bigr) = 0 \text{ in } C_0^\circ,
\end{equation}
we may rewrite \eqref{eq:firstterm} as
\begin{equation}
	\inner{u-u_\epsilon,P(u-u_\epsilon)}_{\lambda,\mu} =  \epsilon \inner{(\lambda_0(\nabla\cdot u_\epsilon)I+2\mu_0\widehat{\nabla}u_\epsilon)_{\text{int}}n, w}_{\partial C_0},
	\label{eq:firstterm2}
\end{equation}
where $\inner{\,\cdot\,,\,\cdot\,}_{\partial C_0} = \inner{\,\cdot\,,\,\cdot\,}_{H^{-1/2}(\partial C_0)^d,H^{1/2}(\partial C_0)^d}$ is the corresponding dual bracket on $\partial C_0$. Now we can straightforwardly estimate \eqref{eq:firstterm2} using the boundedness of the dual bracket, the trace theorem in $\Omega\setminus C_0$, and the boundedness of $P$:
\begin{align}
	\inner{u-u_\epsilon,P(u-u_\epsilon)}_{\lambda,\mu} &\leq \epsilon\norm{(\lambda_0(\nabla\cdot u_\epsilon)I+2\mu_0\widehat{\nabla}u_\epsilon)_{\text{int}}n}_{H^{-1/2}(\partial C_0)^d}\norm{w}_{H^{1/2}(\partial C_0)^d} \notag\\
	&\leq K\epsilon\norm{(\lambda_0(\nabla\cdot u_\epsilon)I+2\mu_0\widehat{\nabla}u_\epsilon)_{\text{int}}n}_{H^{-1/2}(\partial C_0)^d}\norm{u-u_\epsilon}_{\lambda,\mu}. 
	\label{eq:firstterm3}
\end{align}
Since taking the trace of a normal component is bounded from $H_{\textup{div}}(C_0^\circ;\R^{d\times d})$ to $H^{-1/2}(\partial C_0)^d$, \eqref{eq:pdeinC0}, \eqref{eq:nablabnd}, and \eqref{eq:convbnd2} imply
\begin{align}
	\epsilon\norm{(\lambda_0(\nabla\cdot u_\epsilon)I+2\mu_0\widehat{\nabla}u_\epsilon)_{\text{int}}n}_{H^{-1/2}(\partial C_0)^d} \hspace{-5cm}& \notag\\
	&\leq K\epsilon \bigl(\norm{\lambda_0(\nabla\cdot u_\epsilon)I+2\mu_0\widehat{\nabla}u_\epsilon}_{L^2(C_0)^{d\times d}}^2 + \norm{\nabla\cdot(\lambda_0(\nabla\cdot u_\epsilon)I+2\mu_0\widehat{\nabla}u_\epsilon)}_{L^2(C_0)^{d\times d}}^2 \bigr)^{1/2} \notag\\
	&= K\epsilon \norm{\lambda_0(\nabla\cdot u_\epsilon)I+2\mu_0\widehat{\nabla}u_\epsilon}_{L^2(C_0)^{d\times d}} \notag\\
	&\leq K\epsilon \norm{\widehat{\nabla}u_\epsilon}_{L^2(C_0)^{d\times d}} \notag \\
	&\leq K\epsilon^{1/2}\norm{g}. \label{eq:firstterm4}
\end{align}
Combining \eqref{eq:firstterm3} and \eqref{eq:firstterm4} handles the first term on the right-hand side of \eqref{eq:est_u}. Using equivalences of norms, we have altogether proven that
\begin{equation} \label{eq:finalbnd1}
	\norm{u-u_\epsilon}_{H^1(\Omega\setminus C_0)^d} \leq K\epsilon^{1/2}\norm{g}.
\end{equation}

Recall Definition~\ref{def:extension} for the extension $Eu$. Because of \eqref{eq:pdeinC0}, the difference $v = (Eu-u_\epsilon)|_{C_0}$ satisfies the Dirichlet problem
\begin{align*}
	\nabla\cdot\bigl(\lambda_0(\nabla\cdot v)I+2\mu_0\widehat{\nabla}v\bigr) &= 0 \text{ in } C_0^\circ, \\
	v &= u-u_\epsilon \text{ on } \partial C_0.
\end{align*}
Hence, by continuous dependence on the Dirichlet boundary data, the trace theorem in $\Omega\setminus C_0$, and \eqref{eq:finalbnd1}, we get
\begin{equation*}
	\norm{Eu-u_\epsilon}_{H^1(C_0^\circ)^d} \leq K\norm{u-u_\epsilon}_{H^{1/2}(\partial C_0)^d} \leq K\norm{u-u_\epsilon}_{H^1(\Omega\setminus C_0)^d} \leq K\epsilon^{1/2}\norm{g}.
\end{equation*}
Combined with \eqref{eq:finalbnd1}, this proves the first part of the theorem.

The second part of the theorem follows from the trace theorem in $\Omega\setminus C_0$ and \eqref{eq:finalbnd1}:
\begin{equation*}
	\norm{\Lambda(\lambda,\mu)-\Lambda(\lambda_\epsilon,\mu_\epsilon)}_{\mathscr{L}(L^2(\GammaN)^d)} = \sup_{\norm{g}=1}\norm{u-u_\epsilon} \leq K\sup_{\norm{g}=1}\norm{u-u_\epsilon}_{H^1(\Omega\setminus C_0)^d} \leq K\epsilon^{1/2}.
\end{equation*}

\section{Extreme monotonicity inequalities} \label{sec:monoineq}

In this section, we assume $C_0$ and $C_\infty$ (possibly empty sets), with $C = C_0\cup C_\infty$, satisfy Assumption~\ref{assump}, and let $g\in L^2(\GammaN)^d$. 

First, we give natural generalizations of the non-extreme monotonicity inequalities presented in~\cite{Eberle2021,Eberle2021a}.
\begin{lemma} \label{lemma:monoback}
	For $j\in\{1,2\}$, let $\lambda_j,\mu_j\in L_+^\infty(\Omega)$ and 
	\begin{equation*} 
		\widehat{\lambda}_j = \begin{cases}
			\lambda_j & \text{in } \Omega\setminus C\\
			0 & \text{in } C_0\\
			\infty & \text{in } C_\infty
		\end{cases}
		\quad \text{and}
		\quad
		\widehat{\mu}_j = \begin{cases}
			\mu_j & \text{in } \Omega\setminus C\\
			0 & \text{in } C_0\\
			\infty & \text{in } C_\infty.
		\end{cases}
	\end{equation*}
	Then for $\Lambda_j = \Lambda(\widehat{\lambda}_j,\widehat{\mu}_j)$ and $u_j = u_g^{\widehat{\lambda}_j,\widehat{\mu}_j}$, we have
	\begin{align}
		\inner{(\Lambda_2-\Lambda_1)g,g} &\leq \int_{\Omega\setminus C} (\lambda_1-\lambda_2)\abs{\nabla\cdot u_2}^2 + 2(\mu_1-\mu_2)\frob{\widehat{\nabla}u_2}^2\,\di x, \label{eq:monoback1}\\  
		\inner{(\Lambda_2-\Lambda_1)g,g} &\geq \int_{\Omega\setminus C} \frac{\lambda_2}{\lambda_1}(\lambda_1-\lambda_2)\abs{\nabla\cdot u_2}^2 + 2\frac{\mu_2}{\mu_1}(\mu_1-\mu_2)\frob{\widehat{\nabla}u_2}^2 \, \di x. \label{eq:monoback2}
	\end{align}
\end{lemma}
\begin{proof}
	Throughout this proof, we use \eqref{eq:nablabnd} to remove contributions from divergence terms on $C_\infty$ where the symmetric gradients of the displacement fields vanish.
	
	Using the weak forms, we have
	\begin{equation*}
		\inner{u_1,u_2}_{\widehat{\lambda}_1,\widehat{\mu}_1} = \inner{\Lambda_2 g,g} = \norm{u_2}_{\widehat{\lambda}_2,\widehat{\mu}_2}^2.
	\end{equation*}
	Thus
	\begin{align}
		0 &\leq \norm{u_1-u_2}_{\widehat{\lambda}_1,\widehat{\mu}_1}^2 \notag\\
		&= \norm{u_1}_{\widehat{\lambda}_1,\widehat{\mu}_1}^2 + \norm{u_2}_{\widehat{\lambda}_1,\widehat{\mu}_1}^2 - 2\inner{u_1,u_2}_{\widehat{\lambda}_1,\widehat{\mu}_1} \notag\\
		&=\inner{\Lambda_1 g,g} - \inner{\Lambda_2 g,g} + \norm{u_2}_{\widehat{\lambda}_1,\widehat{\mu}_1}^2-\norm{u_2}_{\widehat{\lambda}_2,\widehat{\mu}_2}^2, \label{eq:monoback1comp}
	\end{align}
	which gives \eqref{eq:monoback1}, that is,
	\begin{equation*}
		\inner{(\Lambda_2-\Lambda_1)g,g} \leq \norm{u_2}_{\widehat{\lambda}_1,\widehat{\mu}_1}^2-\norm{u_2}_{\widehat{\lambda}_2,\widehat{\mu}_2}^2.
	\end{equation*}
	Next we show \eqref{eq:monoback2}. By swapping the indices in \eqref{eq:monoback1comp}, 
	\begin{align*}
		\inner{(\Lambda_2-\Lambda_1)g,g} &= \norm{u_2-u_1}_{\widehat{\lambda}_2,\widehat{\mu}_2}^2  +\norm{u_1}_{\widehat{\lambda}_1,\widehat{\mu}_1}^2-\norm{u_1}_{\widehat{\lambda}_2,\widehat{\mu}_2}^2 \\
		&= \int_{\Omega\setminus C} \lambda_2\abs{\nabla\cdot u_2}^2+\lambda_1\abs{\nabla\cdot u_1}^2-2\lambda_2(\nabla\cdot u_2)(\nabla \cdot u_1)\,\di x \\
		&\phantom{={}}+\int_{\Omega\setminus C} 2\mu_2\frob{\widehat{\nabla}u_2}^2+2\mu_1\frob{\widehat{\nabla}u_1}^2-4\mu_2\widehat{\nabla}u_2:\widehat{\nabla}u_1 \,\di x \\
		&=\int_{\Omega\setminus C} \frac{\lambda_2}{\lambda_1}(\lambda_1-\lambda_2)\abs{\nabla\cdot u_2}^2 + 2\frac{\mu_2}{\mu_1}(\mu_1-\mu_2)\frob{\widehat{\nabla}u_2}^2   \,\di x \\
		&\phantom{={}} +\int_{\Omega\setminus C} \lambda_1\bigl\vert\nabla\cdot u_1 - \frac{\lambda_2}{\lambda_1}\nabla\cdot u_2\bigr\vert^2 + 2\mu_1\bigl\Vert\widehat{\nabla}u_1 - \frac{\mu_2}{\mu_1}\widehat{\nabla}u_2\bigr\Vert_{\textup{F}}^2\,\di x \\
		&\geq \int_{\Omega\setminus C} \frac{\lambda_2}{\lambda_1}(\lambda_1-\lambda_2)\abs{\nabla\cdot u_2}^2 + 2\frac{\mu_2}{\mu_1}(\mu_1-\mu_2)\frob{\widehat{\nabla}u_2}^2   \,\di x. \qedhere
	\end{align*}
\end{proof}
In the following two lemmas, consider the Lam\'e parameters $\lambda$ and $\mu$ from \eqref{eq:extremeLame}, with background parameters $\lambda_0,\mu_0\in L^\infty_+(\Omega)$ and extreme inclusions in $C_0$ and $C_\infty$.  We use the notation $\Lambda_{C_0}^{C_\infty}$ for the ND map and set $u_{C_0}^{C_\infty} = u_g^{\lambda,\mu}$.
\begin{lemma} \label{lemma:monoCinfty}
	There is a constant $K>0$ (independent of $g$) such that
	\begin{align} 
		\inner{(\Lambda_{C_0}^{\emptyset}-\Lambda_{C_0}^{C_\infty})g,g} &\leq K\int_{C_\infty} \frob{\widehat{\nabla}u_{C_0}^{\emptyset}}^2\,\di x, \label{eq:monoCinfty1} \\
		\inner{(\Lambda_{C_0}^{\emptyset}-\Lambda_{C_0}^{C_\infty})g,g} & \geq \int_{C_\infty}\lambda_0\abs{\nabla\cdot u_{C_0}^{\emptyset}}^2 +2\mu_0\frob{\widehat{\nabla}u_{C_0}^{\emptyset}}^2\,\di x. \label{eq:monoCinfty2} 
	\end{align}
\end{lemma}
\begin{proof}
	We start by proving \eqref{eq:monoCinfty1}, for which we first note that $u_{C_0}^{C_\infty} = Pu_{C_0}^{\emptyset}$ by Lemma~\ref{lemma:altcharacterization}. Hence, $P^\perp u_{C_0}^{\emptyset} = u_{C_0}^{\emptyset} - u_{C_0}^{C_\infty}$, for which the weak formulation of $u_{C_0}^{\emptyset}$ gives:
	\begin{equation*}
		\inner{(\Lambda_{C_0}^{\emptyset}-\Lambda_{C_0}^{C_\infty})g,g} = \inner{g,P^\perp u_{C_0}^{\emptyset}} = \inner{u_{C_0}^{\emptyset},P^\perp u_{C_0}^{\emptyset}}_{\lambda,\mu} = \norm{P^\perp u_{C_0}^{\emptyset}}_{\lambda,\mu}^2.
	\end{equation*}
	Now the bound in \eqref{eq:monoCinfty1} follows from Lemma~\ref{lemma:perp}.
	
	We proceed to prove \eqref{eq:monoCinfty2}. Define the following truncated (in $C_\infty$) Lam\'e parameters:
	\begin{equation*} 
		\lambda_\epsilon = \begin{cases}
			\lambda_0 & \text{in } \Omega\setminus C\\
			0 & \text{in } C_0\\
			\epsilon^{-1}\lambda_0 & \text{in } C_\infty
		\end{cases}
		\quad \text{and}
		\quad
		\mu_\epsilon = \begin{cases}
			\mu_0 & \text{in } \Omega\setminus C\\
			0 & \text{in } C_0\\
			\epsilon^{-1}\mu_0 & \text{in } C_\infty.
		\end{cases}
	\end{equation*}
	From \eqref{eq:monoback2} in Lemma~\ref{lemma:monoback} (where in this context $\Omega\setminus C$ is replaced by $\Omega\setminus C_0$), we have
	\begin{equation*}
		\inner{(\Lambda_{C_0}^{\emptyset} - \Lambda(\lambda_\epsilon,\mu_\epsilon))g,g} \geq (1-\epsilon)\int_{C_\infty}\lambda_0\abs{\nabla\cdot u_{C_0}^{\emptyset}}^2 +2\mu_0\frob{\widehat{\nabla}u_{C_0}^{\emptyset}}^2\,\di x.
	\end{equation*}
Due to Theorem~\ref{thm:convergence}, the bound in \eqref{eq:monoCinfty2} is obtained by letting $\epsilon\to 0$.
\end{proof}
\begin{lemma} \label{lemma:monoC0}
	We have the bounds
	\begin{align} 
		\inner{(\Lambda_{C_0}^{C_\infty}-\Lambda_{\emptyset}^{C_\infty})g,g} &\leq \int_{C_0} \lambda_0\abs{\nabla\cdot Eu_{C_0}^{C_\infty}}^2 + 2\mu_0\frob{\widehat{\nabla}Eu_{C_0}^{C_\infty}}^2\,\di x, \label{eq:monoC01} \\
		\inner{(\Lambda_{C_0}^{C_\infty}-\Lambda_{\emptyset}^{C_\infty})g,g} &\geq \int_{C_0}\lambda_0\abs{\nabla\cdot u_{\emptyset}^{C_\infty}}^2 +2\mu_0\frob{\widehat{\nabla}u_{\emptyset}^{C_\infty}}^2\,\di x. \label{eq:monoC02}
	\end{align}
\end{lemma}
\begin{proof}
	We start by proving \eqref{eq:monoC01}. Let $\widehat{\lambda}$ and $\widehat{\mu}$ be the Lam\'e parameters associated with $\Lambda_{\emptyset}^{C_\infty}$, recall the extension operator $E$ (Definition~\ref{def:extension}), and note that
	\begin{equation*}
		\inner{u_{\emptyset}^{C_\infty},Eu_{C_0}^{C_\infty}}_{\widehat{\lambda},\widehat{\mu}} = \inner{\Lambda_{C_0}^{C_\infty}g,g} = \norm{Eu_{C_0}^{C_\infty}}_{\widehat{\lambda},\widehat{\mu}}^2 - \int_{C_0} \lambda_0\abs{\nabla\cdot Eu_{C_0}^{C_\infty}}^2 + 2\mu_0\frob{\widehat{\nabla}Eu_{C_0}^{C_\infty}}^2\,\di x.
	\end{equation*}
	The bound in \eqref{eq:monoC01} now follows via a short computation:
	\begin{align*}
		0 & \leq \norm{Eu_{C_0}^{C_\infty}-u_{\emptyset}^{C_\infty}}_{\widehat{\lambda},\widehat{\mu}} \\
		&= \norm{Eu_{C_0}^{C_\infty}}_{\widehat{\lambda},\widehat{\mu}}^2 + \norm{u_{\emptyset}^{C_\infty}}_{\widehat{\lambda},\widehat{\mu}}^2  - 2 \inner{u_{\emptyset}^{C_\infty},Eu_{C_0}^{C_\infty}}_{\widehat{\lambda},\widehat{\mu}} \\
		&= \inner{(\Lambda_{\emptyset}^{C_\infty}-\Lambda_{C_0}^{C_\infty})g,g} + \int_{C_0} \lambda_0\abs{\nabla\cdot Eu_{C_0}^{C_\infty}}^2 + 2\mu_0\frob{\widehat{\nabla}Eu_{C_0}^{C_\infty}}^2\,\di x. 
	\end{align*}
	
	We proceed to prove \eqref{eq:monoC02}. Define the following truncated (in $C_0$) Lam\'e parameters:
	\begin{equation*} 
		\lambda_\epsilon = \begin{cases}
			\lambda_0 & \text{in } \Omega\setminus C\\
			\epsilon\lambda_0 & \text{in } C_0\\
			\infty & \text{in } C_\infty
		\end{cases}
		\quad \text{and}
		\quad
		\mu_\epsilon = \begin{cases}
			\mu_0 & \text{in } \Omega\setminus C\\
			\epsilon\mu_0 & \text{in } C_0\\
			\infty & \text{in } C_\infty.
		\end{cases}
	\end{equation*}
	From \eqref{eq:monoback1} in Lemma~\ref{lemma:monoback} (where in this context $\Omega\setminus C$ is replaced by $\Omega\setminus C_\infty$), we have
	\begin{equation*}
		\inner{(\Lambda(\lambda_\epsilon,\mu_\epsilon)-\Lambda_{\emptyset}^{C_\infty})g,g} \geq (1-\epsilon)\int_{C_0}\lambda_0\abs{\nabla\cdot u_{\emptyset}^{C_\infty}}^2 +2\mu_0\frob{\widehat{\nabla}u_{\emptyset}^{C_\infty}}^2\,\di x.
	\end{equation*}
        Due to Theorem~\ref{thm:convergence}, the bound in \eqref{eq:monoC02} is obtained by letting $\epsilon\to 0$.
\end{proof}

\section{Virtual measurement operators} \label{sec:virtualop}

We will need some results on so-called \emph{virtual measurement operators} for Section~\ref{sec:localized}; see \cite{Eberle2021a,Garde2025,Garde2024,Harrach2008,Harrach2013} for more information on these types of operators. Let $V\subseteq\Omega$ be measurable and define the operators $\widehat{L}_V(\lambda_0,\mu_0) \colon L^2(V)^{d\times d} \to L^2(\GammaN)^d$ and $L_V(\lambda_0,\mu_0) \colon L^2(V) \to L^2(\GammaN)^d$ by
\begin{equation*}
	\widehat{L}_V(\lambda_0,\mu_0)F = \widehat{w}_{F}^{\lambda_0,\mu_0}|_{\GammaN} \quad\text{and}\quad L_V(\lambda_0,\mu_0)G = w_{G}^{\lambda_0,\mu_0}|_{\GammaN}
\end{equation*}
for $\lambda_0,\mu_0\in L^\infty_+(\Omega)$. Here $\widehat{w} = \widehat{w}_{F}^{\lambda_0,\mu_0}$ and $w = w_{G}^{\lambda_0,\mu_0}$ are the unique solutions in $\H_\emptyset^\emptyset$ of the variational problems
\begin{align}
	\inner{\widehat{w},v}_{\lambda_0,\mu_0} &= \inner{F,\widehat{\nabla}v}_{L^2(V)^{d\times d}},  \label{eq:virtopprob1}\\
	\inner{w,v}_{\lambda_0,\mu_0} &= \inner{G,\nabla\cdot v}_{L^2(V)} \label{eq:virtopprob2}
\end{align}
for all $v\in\H_\emptyset^\emptyset$. We have the following properties, with $R$ denoting the \emph{range} of an operator.
\begin{proposition} \needspace{2\baselineskip} \label{prop:virtualop}
	Let $V,V_1,V_2$ be measurable subsets of $\Omega$.
	\begin{enumerate}[\rm(i)]
		\item If $V$ has Lipschitz boundary and $\lambda_1,\mu_1\in L^\infty_+(\Omega)$ satisfy 
		\begin{equation*}
			\suppm(\lambda_0-\lambda_1)\cup\suppm(\mu_0-\mu_1)\subseteq V,
		\end{equation*}
		then
		\begin{equation*}
			R(L_V(\lambda_1,\mu_1))\subseteq R(\widehat{L}_V(\lambda_0,\mu_0)) = R(\widehat{L}_V(\lambda_1,\mu_1)).
		\end{equation*}
		\item $V_1\subseteq V_2$ implies 
		\begin{equation*}
			R(\widehat{L}_{V_1}(\lambda_0,\mu_0))\subseteq R(\widehat{L}_{V_2}(\lambda_0,\mu_0)) \quad\text{and}\quad R(L_{V_1}(\lambda_0,\mu_0))\subseteq R(L_{V_2}(\lambda_0,\mu_0)).
		\end{equation*}
		\item For $g\in L^2(\GammaN)^d$, the adjoints satisfy
		\begin{equation*}
			\widehat{L}_V(\lambda_0,\mu_0)^*g = \widehat{\nabla}u_g^{\lambda_0,\mu_0}|_V \quad\text{and}\quad L_V(\lambda_0,\mu_0)^*g = \nabla\cdot u_g^{\lambda_0,\mu_0}|_V.
		\end{equation*} 
	\end{enumerate}
\end{proposition}
\begin{proof} 
	\textbf{(i):} Let there be another pair of Lam\'e parameters $\lambda_1,\mu_1\in L^\infty_+(\Omega)$, which only may differ from $\lambda_0$ and $\mu_0$ in $V$. Let $\varphi \in R(\widehat{L}_V(\lambda_0,\mu_0))$, i.e.,
	\begin{equation*}
		\varphi = \widehat{L}_V(\lambda_0,\mu_0)F_0 = w_0|_{\GammaN},
	\end{equation*}
	for some $F_0\in L^2(V)^{d\times d}$ and $w_0 = \widehat{w}_{F_0}^{\lambda_0,\mu_0}$. In particular,
	\begin{align}
		\inner{w_0,v}_{\lambda_1,\mu_1} &= \int_{\Omega}(\lambda_1-\lambda_0)(\nabla\cdot w_0)(\nabla\cdot v) + 2(\mu_1-\mu_0)\widehat{\nabla}w_0:\widehat{\nabla}v\,\di x + \inner{w_0,v}_{\lambda_0,\mu_0} \notag \\
		&= \int_V (\lambda_1-\lambda_0)(\nabla\cdot w_0)(\nabla\cdot v) + 2(\mu_1-\mu_0)\widehat{\nabla}w_0:\widehat{\nabla}v\,\di x + \inner{F_0,\widehat{\nabla}v}_{L^2(V)^{d\times d}}. \label{eq:virtopproof1}
	\end{align}
	In the following, we may assume that $V$ consists of a single connected component because otherwise the integral in \eqref{eq:virtopproof1} can be separately considered on each connected component, as can the rest of the proof. 
	
Consider the Hilbert space 
\begin{equation}
  \label{eq:VHilbert}
		\mathcal{V} = \Bigl\{\, v\in H^1(V^\circ) : \int_V v\,\di x = 0 \text{ in }\R^d \,\Bigr\}
	\end{equation}
        and the auxiliary variational problem
	\begin{equation} \label{eq:hvarprob}
		\inner{\widehat{\nabla}h,\widehat{\nabla}v}_{L^2(V)^{d\times d}} = \int_V (\lambda_1-\lambda_0)(\nabla\cdot w_0)(\nabla\cdot v) + 2(\mu_1-\mu_0)\widehat{\nabla}w_0:\widehat{\nabla}v\,\di x
	\end{equation}
        for all $v\in \mathcal{V}$. The zero-mean condition in \eqref{eq:VHilbert}, combined with Poincar\'e's and Korn's inequalities, ensures that the left-hand side of \eqref{eq:hvarprob} defines on $\mathcal{V}$ an inner product, which induces a norm that is equivalent to the usual $H^1$-norm on $\mathcal{V}$.
	%
%	\begin{equation} \label{eq:hvarprob}
%		\inner{\widehat{\nabla}h,\widehat{\nabla}v}_{L^2(V)^{d\times d}} = \int_V (\lambda_1-\lambda_0)(\nabla\cdot w_0)(\nabla\cdot v) + 2(\mu_1-\mu_0)\widehat{\nabla}w_0:\widehat{\nabla}v\,\di x,
%	\end{equation}
	%
        %	for all $v\in \mathcal{V}$.
        In consequence, the Lax--Milgram lemma ensures a unique solution $h\in \mathcal{V}$ to \eqref{eq:hvarprob}. 
	
	Returning to \eqref{eq:virtopproof1}, we can assume that the considered $v$ satisfy $v|_V\in \mathcal{V}$, since the values of the expressions on the right-hand side of \eqref{eq:virtopproof1} do not change if a constant vector is subtracted from~$v$. Hence we have
	\begin{equation*}
		\inner{w_0,v}_{\lambda_1,\mu_1} = \inner{\widehat{\nabla}h + F_0,\widehat{\nabla}v}_{L^2(V)^{d\times d}},
	\end{equation*} 
and the unique solvability of \eqref{eq:virtopprob1} implies
	\begin{equation*}
		\varphi = w_0|_{\GammaN} = \widehat{L}_V(\lambda_1,\mu_1)F_1
	\end{equation*} 
        for $F_1 = \widehat{\nabla}h + F_0$. Thereby, we have proven $R(\widehat{L}_V(\lambda_0,\mu_0))\subseteq R(\widehat{L}_V(\lambda_1,\mu_1))$. The opposite set inclusion follows from the same proof, by swapping the roles of $(\lambda_0,\mu_0)$ and $(\lambda_1,\mu_1)$.
	
	Now for the proof involving $L_V$, we immediately get $R(L_V(\lambda_1,\mu_1))\subseteq R(\widehat{L}_V(\lambda_1,\mu_1))$ from part~(iii) of this proposition, \eqref{eq:nablabnd}, and Lemma~\ref{lemma:rangenorm}.
	
	\textbf{(ii):} We only give the proof for $\widehat{L}_V(\lambda_0,\mu_0)$ since the one for $L_V(\lambda_0,\mu_0)$ is almost identical. Let $\varphi = \widehat{L}_{V_1}(\lambda_0,\mu_0)F_1 = \widehat{w}|_{\GammaN}$ for some $F_1\in L^2(V_1)^{d\times d}$ and $\widehat{w} = \widehat{w}_{F_1}^{\lambda_0,\mu_0}$. As $V_1\subseteq V_2$, we can extend $F_1$ by zero to $F_2\in L^2(V_2)^{d\times d}$, and it immediately follows that
	\begin{equation*}
		\inner{\widehat{w},v}_{\lambda_0,\mu_0} = \inner{F_1,\widehat{\nabla}v}_{L^2(V_1)^{d\times d}} = \inner{F_2,\widehat{\nabla}v}_{L^2(V_2)^{d\times d}}
	\end{equation*}
        for all $v\in\H_\emptyset^\emptyset$. Hence, $\varphi = \widehat{L}_{V_2}(\lambda_0,\mu_0)F_2$, and therefore $R(\widehat{L}_{V_1}(\lambda_0,\mu_0))\subseteq R(\widehat{L}_{V_2}(\lambda_0,\mu_0))$.
	
	\textbf{(iii):} Again, we only present the proof for $\widehat{L}_V(\lambda_0,\mu_0)$ since the one for $L_V(\lambda_0,\mu_0)$ is almost identical. The result follows from a short calculation that utilizes the variational problem \eqref{eq:virtopprob1} for $\widehat{w} = \widehat{w}_{F}^{\lambda_0,\mu_0}$ and the weak problem \eqref{eq:weakform} for $u = u_g^{\lambda_0,\mu_0}$:
	\begin{align*}
		\inner{\widehat{L}_V(\lambda_0,\mu_0)^* g,F}_{L^2(V)^{d\times d}} &= \inner{g,\widehat{L}_V(\lambda_0,\mu_0)F} =\inner{g,\widehat{w}} \\
		&=\inner{u,\widehat{w}}_{\lambda_0,\mu_0} =\inner{\widehat{\nabla}u,F}_{L^2(V)^{d\times d}},
	\end{align*}
	and thus $\widehat{L}_V(\lambda_0,\mu_0)^* g = \widehat{\nabla}u|_V$.
\end{proof}
\begin{remark}
	Note that we only have the equality of ranges for $\widehat{L}$ in Proposition~\ref{prop:virtualop}(i). The natural approach to proving the range equality also for $L$ would require solving an auxiliary variational problem like \eqref{eq:hvarprob}, but with the left hand-side $\inner{\nabla\cdot h,\nabla \cdot v}_{L^2(V)}$, which clearly does not define an inner product on $\H_\emptyset^\emptyset$. One can try to fix this flaw by only considering \eqref{eq:virtopprob2} for $v\in\mathcal{W}^\perp$, where $\mathcal{W}\subset\H_{\emptyset}^{\emptyset}$ are the functions with $\nabla\cdot v\equiv 0$ and the orthogonal complement $\mathcal{W}^\perp$ is with respect to $\inner{\,\cdot\,,\,\cdot\,}_{\lambda_0,\mu_0}$. This construction still gives $L_V(\lambda_0,\mu_0)^*g = \nabla\cdot u_g^{\lambda_0,\mu_0}|_V$. However, while we get an inner product, its norm is not equivalent to the $H^1$-norm, so we lack the Hilbert space property needed for the Lax--Milgram lemma. 
\end{remark}
The range inclusions can be translated into norm estimates for the adjoints.
\begin{lemma}[Lemma~2.5 in \cite{Harrach2008}] \label{lemma:rangenorm}
	Let $H$, $K_1$, and $K_2$ be Hilbert spaces and let $A_j\in\mathscr{L}(K_j,H)$ for $j = 1,2$. Then
	\begin{equation*}
		R(A_1) \subseteq R(A_2) \qquad \text{if and only if} \qquad \exists K>0, \forall x\in H\colon \norm{A_1^* x}_{K_1} \leq K\norm{A_2^* x}_{K_2}.
	\end{equation*}
\end{lemma}

\section{Localization results} \label{sec:localized}

In this section, we assume $C_0$ and $C_\infty$, with $C = C_0\cup C_\infty$, satisfy Assumption~\ref{assump}, and let $\lambda$ and $\mu$ be as in \eqref{eq:extremeLame}, with background parameters $\lambda_0,\mu_0\in L^\infty_+(\Omega)$ and extreme inclusions in $C$.

One of the core ingredients in proving the ``difficult direction'' for monotonicity-based methods, such as the latter statements in Theorems~\ref{thm:monoext}--\ref{thm:mononeg} and \ref{thm:monofinitelin}, is localized potentials. We start by the localization result in the non-extreme case. Recall Definition~\ref{def:ucp} and Remark~\ref{remark:ucp} on the UCP condition.
\begin{lemma}[Theorem~3.3 in \cite{Eberle2021a}] \label{lemma:locpot}
	Let $U\subset\overline{\Omega}$ be a relatively open connected set that intersects $\GammaN$. Let $B\Subset U$ be a non-empty open set and assume $\lambda_0$ and $\mu_0$ satisfy the UCP in $U$. Then there are sequences $(g_i)$ in $L^2(\GammaN)^d$ and $(u_i)$ in $H^1(\Omega)^d$, with $u_i = u_{g_i}^{\lambda_0,\mu_0}$, such that
	\begin{align}
		\lim_{i\to\infty}\int_B \abs{\nabla\cdot u_i}^2\,\di x = \lim_{i\to\infty}\int_B \frob{\widehat{\nabla}u_i}^2\,\di x &= \infty, \label{eq:locpotblowup}\\
		\lim_{i\to\infty}\int_{\Omega\setminus U} \abs{\nabla\cdot u_i}^2\,\di x = \lim_{i\to\infty}\int_{\Omega\setminus U} \frob{\widehat{\nabla}u_i}^2\,\di x &= 0. \label{eq:locpotzero}
	\end{align}
\end{lemma}
\begin{proof}
	This result is a reformulation of \cite[Theorem~3.3]{Eberle2021a}, for which the same proof can be completed using $D_1 = B$ and $D_2 = \Omega\setminus\overline{U}$. For the same approach, in the context of electrical impedance tomography, see also \cite[Proof of Theorem~3.6]{Harrach2013}.
\end{proof}
First we prove that such localization can partially be transferred to another set of non-extreme Lam\'e parameters, provided that the differences in the two sets of Lam\'e parameters occur in the part of the domain where the localized potentials tend to zero.
\begin{lemma} \label{lemma:locpottransf}
	Let $U$ and $B$ be as in Lemma~\ref{lemma:locpot}, and such that $\partial(\Omega\setminus\overline{U})$ and $\partial B$ are Lipschitz continuous. Suppose $u_i = u_{g_i}^{\lambda_0,\mu_0}$ satisfies \eqref{eq:locpotblowup} and \eqref{eq:locpotzero}. Assume that $\lambda_1,\mu_1\in L^\infty_+(\Omega)$ and $\suppm(\lambda_0-\lambda_1)\cup\suppm(\mu_0-\mu_1)\subset \Omega\setminus\overline{U}$. Define $\widehat{u}_i = u_{g_i}^{\lambda_1,\mu_1}$. Then
	\begin{align*}
		\lim_{i\to\infty}\int_B \frob{\widehat{\nabla}\widehat{u}_i}^2\,\di x &= \infty, \\
		\lim_{i\to\infty}\int_{\Omega\setminus U} \abs{\nabla\cdot \widehat{u}_i}^2\,\di x &= \lim_{i\to\infty}\int_{\Omega\setminus U} \frob{\widehat{\nabla}\widehat{u}_i}^2\,\di x = 0.
	\end{align*}
\end{lemma}
\begin{proof}
	Let $V_1 = \Omega\setminus\overline{U}$ and $V_2 = B$. Using Proposition~\ref{prop:virtualop}(i), we have
	\begin{equation} \label{eq:range1}
		R(L_{V_1}(\lambda_1,\mu_1)) \subseteq R(\widehat{L}_{V_1}(\lambda_0,\mu_0)) = R(\widehat{L}_{V_1}(\lambda_1,\mu_1)),
	\end{equation}
	and
	\begin{equation} \label{eq:range2}
		R(\widehat{L}_{V_1\cup V_2}(\lambda_0,\mu_0)) = R(\widehat{L}_{V_1\cup V_2}(\lambda_1,\mu_1)).
	\end{equation}
	Using Proposition~\ref{prop:virtualop}(iii) gives
	\begin{equation*}
		\lim_{i\to\infty}\norm{\widehat{L}_{V_1}(\lambda_0,\mu_0)^*g_i}_{L^2(V_1)^{d\times d}}^2 = \lim_{i\to\infty}\int_{\Omega\setminus U} \frob{\widehat{\nabla}u_i}^2\,\di x = 0,
	\end{equation*}	
	which by \eqref{eq:range1} and Lemma~\ref{lemma:rangenorm} implies
	\begin{align}
		\lim_{i\to\infty}\int_{\Omega\setminus U} \frob{\widehat{\nabla}\widehat{u}_i}^2\,\di x &= \lim_{i\to\infty}\norm{\widehat{L}_{V_1}(\lambda_1,\mu_1)^*g_i}_{L^2(V_1)^{d\times d}}^2 = 0, \label{eq:range3} \\
		\lim_{i\to\infty}\int_{\Omega\setminus U} \abs{\nabla\cdot \widehat{u}_i}^2\,\di x &= \lim_{i\to\infty}\norm{L_{V_1}(\lambda_1,\mu_1)^*g_i}_{L^2(V_1)}^2 = 0. \notag
	\end{align}
	Moreover, \eqref{eq:range2}, Proposition~\ref{prop:virtualop}(iii), and Lemma~\ref{lemma:rangenorm} give
	\begin{align*}
		\norm{\widehat{L}_{V_1}(\lambda_1,\mu_1)^*g_i}_{L^2(V_1)^{d\times d}}^2 + \norm{\widehat{L}_{V_2}(\lambda_1,\mu_1)^*g_i}_{L^2(V_2)^{d\times d}}^2 & =  \norm{\widehat{L}_{V_1\cup V_2}(\lambda_1,\mu_1)^*g_i}_{L^2(V_1\cup V_2)^{d\times d}}^2 \\
		&\geq K\norm{\widehat{L}_{V_1\cup V_2}(\lambda_0,\mu_0)^*g_i}_{L^2(V_1\cup V_2)^{d\times d}}^2 \\
		&\geq K\norm{\widehat{L}_{V_2}(\lambda_0,\mu_0)^*g_i}_{L^2(V_2)^{d\times d}}^2,
	\end{align*}
	where the lower bound tends to $\infty$ for $i\to\infty$. Combined with \eqref{eq:range3} and Proposition~\ref{prop:virtualop}(iii), this yields
	\begin{equation*}
		\lim_{i\to\infty}\int_{B} \frob{\widehat{\nabla}\widehat{u}_i}^2\,\di x = \lim_{i\to\infty}\norm{\widehat{L}_{V_2}(\lambda_1,\mu_1)^*g_i}_{L^2(V_2)^{d\times d}}^2 = \infty. \qedhere
	\end{equation*}
\end{proof}
Next we show that the localization of Lemma~\ref{lemma:locpot} can be fully transferred to the case with extreme inclusions, provided the extreme inclusions are in the part of the domain where the localized potentials tend to zero. Interestingly, the result is stronger than for finite perturbations.
\begin{lemma} \label{lemma:locpotextreme}
	Let $U$ and $B$ be as in Lemma~\ref{lemma:locpot}. Suppose $u_i = u_{g_i}^{\lambda_0,\mu_0}$ satisfies \eqref{eq:locpotblowup} and \eqref{eq:locpotzero}. Assume that $C \subset \Omega\setminus\overline{U}$ and define $\widehat{u}_i = Eu_{g_i}^{\lambda,\mu}$. Then
	\begin{align*}
		\lim_{i\to\infty}\int_B \abs{\nabla\cdot \widehat{u}_i}^2\,\di x = \lim_{i\to\infty}\int_B \frob{\widehat{\nabla}\widehat{u}_i}^2\,\di x &= \infty, \\
		\lim_{i\to\infty}\int_{\Omega\setminus U} \abs{\nabla\cdot \widehat{u}_i}^2\,\di x = \lim_{i\to\infty}\int_{\Omega\setminus U} \frob{\widehat{\nabla}\widehat{u}_i}^2\,\di x &= 0.
	\end{align*}
\end{lemma}
\begin{proof}
	We start by using Lemma~\ref{lemma:altcharacterization} to write
	\begin{equation*}
		\widehat{v}_i = u_{g_i}^{\lambda,\mu} = Pv_i,
	\end{equation*}
	with $v_i = u_i|_{\Omega\setminus C_0} - w_i$ and $w_i\in H^1(\Omega\setminus C_0)^d$ being the solution to the auxiliary problem in Lemma~\ref{lemma:altcharacterization} (with $g$ replaced by $g_i$).
	
	To respect jump-conditions across $\partial C_0$, we use ``ext'' to indicate a trace taken on $\partial C_0$ from within $\Omega\setminus C_0$. We now pick an open set $V$ such that $C_0 \subset \Omega\setminus\overline{V}\subset \Omega\setminus\overline{U}$ and such that $\partial(\Omega\setminus\overline{V})$ is Lipschitz continuous (which is not guaranteed by using $U$). By the continuous dependence of $w_i$ on its Neumann trace and a bound similar to \eqref{eq:firstterm4} (involving $H_\textup{div}(\Omega\setminus(\overline{V}\cup C_0);\R^{d\times d})$ instead of $H_\textup{div}(C_0^\circ;\R^{d\times d})$),
	\begin{align}
		\norm{\widehat{\nabla}w_i}_{L^2(\Omega\setminus C_0)^{d\times d}} &\leq K\norm{w_i}_{H^1(\Omega\setminus C_0)^d} \notag\\
		&\leq K\norm{\bigl(\lambda_0(\nabla\cdot u_i)I+2\mu_0\widehat{\nabla}u_i\bigr)_{\textup{ext}}n}_{H^{-1/2}(\partial C_0)^d} \notag\\
	%	&\leq K\norm{\widehat{\nabla}u_i}_{L^2(\Omega\setminus(V\cup C_0))^{d\times d}} \notag\\
		&\leq K\norm{\widehat{\nabla}u_i}_{L^2(\Omega\setminus(U\cup C_0))^{d\times d}} \to 0 \text{ for } i\to\infty, \label{eq:locpotext1}
	\end{align}
	where the limit is a consequence of \eqref{eq:locpotzero}. Hence, by \eqref{eq:nablabnd} and \eqref{eq:locpotext1}, the limiting behaviors of $\nabla\cdot v_i$ and $\widehat{\nabla}v_i$ coincide with those of $\nabla\cdot u_i|_{\Omega\setminus C_0}$ and $\widehat{\nabla}u_i|_{\Omega\setminus C_0}$. From \eqref{eq:locpotblowup} and \eqref{eq:locpotzero}, we thus obtain
	\begin{align}
		\lim_{i\to\infty}\int_B \abs{\nabla\cdot v_i}^2\,\di x = \lim_{i\to\infty}\int_B \frob{\widehat{\nabla}v_i}^2\,\di x &= \infty, \label{eq:locpotext2} \\
		\lim_{i\to\infty}\int_{\Omega\setminus(U\cup C_0)} \abs{\nabla\cdot v_i}^2\,\di x = \lim_{i\to\infty}\int_{\Omega\setminus(U\cup C_0)} \frob{\widehat{\nabla}v_i}^2\,\di x &= 0.  \label{eq:locpotext3}
	\end{align}
	Using Lemma~\ref{lemma:perp},
	\begin{equation}
		\norm{\widehat{\nabla}P^\perp v_i}_{L^2(\Omega\setminus C_0)^{d\times d}} \leq K\norm{P^\perp v_i}_{\lambda,\mu} \leq K\norm{\widehat{\nabla}v_i}_{L^2(C_\infty)^{d\times d}}\to 0 \text{ for } i\to\infty, \label{eq:locpotext4}
	\end{equation}
	where the limit is due to \eqref{eq:locpotext3} and the set inclusion $C_\infty\subset \Omega\setminus(\overline{U}\cup C_0)$. Since $\widehat{v}_i = v_i - P^\perp v_i$, the equations \eqref{eq:locpotext2}--\eqref{eq:locpotext4} imply
	\begin{align*}
		\lim_{i\to\infty}\int_B \abs{\nabla\cdot \widehat{v}_i}^2\,\di x = \lim_{i\to\infty}\int_B \frob{\widehat{\nabla}\widehat{v}_i}^2\,\di x &= \infty, \\
		\lim_{i\to\infty}\int_{\Omega\setminus(U\cup C_0)} \abs{\nabla\cdot \widehat{v}_i}^2\,\di x = \lim_{i\to\infty}\int_{\Omega\setminus(U\cup C_0)} \frob{\widehat{\nabla}\widehat{v}_i}^2\,\di x &= 0.  
	\end{align*}
	As $\widehat{u}_i = E\widehat{v}_i$, it holds that $\widehat{u}_i|_{\Omega\setminus C_0} = \widehat{v}_i$. What remains to be investigated is the behavior of $\widehat{u}_i|_{C_0}$.
	
	As $\widehat{u}_i|_{\Omega\setminus(\overline{U}\cup C_0)}$ tends to zero in the quotient space $(H^1(\Omega\setminus(\overline{U}\cup C_0))/\R)^d$, the trace $\widehat{u}_i|_{\partial C_0}$ tends to zero in $(H^{1/2}(\partial C_0)/\R)^d$. We also have $\nabla\widehat{u}_i = \nabla\widetilde{u}_i$ in $C_0^\circ$, where $\widetilde{u}_i$ is the unique solution in $(H^1(C_0^\circ)/\R)^d$ to 
	\begin{align*}
		\nabla\cdot \bigl(\lambda_0(\nabla\cdot \widetilde{u}_i)I+2\mu_0\widehat{\nabla}\widetilde{u}_i\bigr) &= 0 \text{ in } C_0^\circ, \\
		\widetilde{u}_i &= \widehat{u}_i \text{ on } \partial C_0.
	\end{align*}
	Here the Dirichlet data in $(H^{1/2}(\partial C_0)/\R)^d$ is the equivalence class identified by the element $\widehat{u}_i|_{\partial C_0} \in H^{1/2}(\partial C_0)^d$. By the continuous dependence on the Dirichlet data, $\nabla \widehat{u}_i|_{C_0^\circ}$ thus tends to zero in $L^2$, implying that
	\begin{equation*}
		\lim_{i\to\infty}\int_{C_0} \abs{\nabla\cdot \widehat{u}_i}^2\,\di x = \lim_{i\to\infty}\int_{C_0} \frob{\widehat{\nabla}\widehat{u}_i}^2\,\di x = 0. \qedhere
	\end{equation*}
\end{proof}
\begin{remark} \label{remark:loc}
	Suppose we have just the results
	\begin{align*}
		\lim_{i\to\infty}\int_B \frob{\widehat{\nabla}u_i}^2\,\di x &= \infty, \\
		\lim_{i\to\infty}\int_{\Omega\setminus U} \abs{\nabla\cdot u_i}^2\,\di x &= \lim_{i\to\infty}\int_{\Omega\setminus U} \frob{\widehat{\nabla}u_i}^2\,\di x = 0
	\end{align*}
	at our disposal in the proof of Lemma~\ref{lemma:locpotextreme}. By the same proof, we still get
	\begin{align*}
		\lim_{i\to\infty}\int_B \frob{\widehat{\nabla}\widehat{u}_i}^2\,\di x &= \infty, \\
		\lim_{i\to\infty}\int_{\Omega\setminus U} \abs{\nabla\cdot \widehat{u}_i}^2\,\di x &= \lim_{i\to\infty}\int_{\Omega\setminus U} \frob{\widehat{\nabla}\widehat{u}_i}^2\,\di x = 0.
	\end{align*}
	Thus we can simultaneously localize for a chain of coefficients: First use Lemma~\ref{lemma:locpot}, then Lemma~\ref{lemma:locpottransf} to localize for the coefficients $(\lambda_1,\mu_1)$ (thus lacking the blow-up in $B$ for the divergence), and finally Lemma~\ref{lemma:locpotextreme} with $(\lambda_1,\mu_1)$ in place of $(\lambda_0,\mu_0)$ both in the statement of the lemma and in the definition of the pair $(\lambda,\mu)$.
\end{remark}

\section{Proof of Theorem~\ref{thm:monoext}} \label{sec:monoextproof}

Recall from Definition~\ref{def:posneginc} that $D = \Dm\cup\Dp \cup D_0\cup D_\infty$ and $\lambda$ and $\mu$ are on the form
\begin{equation*} 
	\lambda = \begin{cases}
		\lambda_0 & \text{in } \Omega\setminus D \\
		\lambda_- & \text{in } \Dm \\
		\lambda_+ & \text{in } \Dp \\
		0 & \text{in } D_0 \\
		\infty & \text{in } D_\infty
	\end{cases}
	\quad\text{and}\quad
	\mu = \begin{cases}
		\mu_0 & \text{in } \Omega\setminus D \\
		\mu_- & \text{in } \Dm \\
		\mu_+ & \text{in } \Dp \\
		0 & \text{in } D_0 \\
		\infty & \text{in } D_\infty,
	\end{cases}
\end{equation*}
where $\lambda_- \leq \lambda_0 \leq \lambda_+$ and $\mu_-\leq\mu_0\leq\mu_+$ (all functions in $L^\infty_+(\Omega)$) and satisfy Assumption~\ref{assump:technical}.

\subsection*{Proof of ``$\boldsymbol{D\subseteq C \Rightarrow \Lambda_{C}^{\emptyset} \geq \Lambda(\lambda,\mu) \geq \Lambda_{\emptyset}^{C}}$''}

Consider the truncated Lam\'e parameters:
\begin{equation*} 
	\lambda_\epsilon = \begin{cases}
		\lambda_0 & \text{in } \Omega\setminus D \\
		\lambda_- & \text{in } \Dm \\
		\lambda_+ & \text{in } \Dp \\
		\epsilon\lambda_0 & \text{in } D_0 \\
		\epsilon^{-1}\lambda_0 & \text{in } D_\infty
	\end{cases}
	\quad\text{and}\quad
	\mu_\epsilon = \begin{cases}
		\mu_0 & \text{in } \Omega\setminus D \\
		\mu_- & \text{in } \Dm \\
		\mu_+ & \text{in } \Dp \\
		\epsilon\mu_0 & \text{in } D_0 \\
		\epsilon^{-1}\mu_0 & \text{in } D_\infty,
	\end{cases}
\end{equation*}
and likewise for the test operators:
\begin{equation*} 
	\widehat{\lambda}_\epsilon = \begin{cases}
		\lambda_0 & \text{in } \Omega\setminus C \\
		\epsilon\lambda_0 & \text{in } C
	\end{cases}
	\quad\text{and}\quad
	\widehat{\mu}_\epsilon = \begin{cases}
		\mu_0 & \text{in } \Omega\setminus C \\
		\epsilon\mu_0 & \text{in } C,
	\end{cases}
\end{equation*}
and
\begin{equation*} 
	\widetilde{\lambda}_\epsilon = \begin{cases}
		\lambda_0 & \text{in } \Omega\setminus C \\
		\epsilon^{-1}\lambda_0 & \text{in } C
	\end{cases}
	\quad\text{and}\quad
	\widetilde{\mu}_\epsilon = \begin{cases}
		\mu_0 & \text{in } \Omega\setminus C \\
		\epsilon^{-1}\mu_0 & \text{in } C.
	\end{cases}
\end{equation*}
As $D\subseteq C$, for small enough $\epsilon>0$ we have $\widehat{\lambda}_\epsilon\leq\lambda_\epsilon\leq \widetilde{\lambda}_\epsilon$ and $\widehat{\mu}_\epsilon\leq\mu_\epsilon\leq \widetilde{\mu}_\epsilon$. Lemma~\ref{lemma:monoback} thus implies $\Lambda(\widehat{\lambda}_\epsilon,\widehat{\mu}_\epsilon) \geq \Lambda(\lambda_\epsilon,\mu_\epsilon) \geq \Lambda(\widetilde{\lambda}_\epsilon,\widetilde{\mu}_\epsilon)$. Now the operator norm convergence for $\epsilon\to 0$ from Theorem~\ref{thm:convergence} gives $\Lambda_C^\emptyset \geq \Lambda(\lambda,\mu) \geq \Lambda_{\emptyset}^{C}$.

\subsection*{Proof of ``$\boldsymbol{\Lambda_{C}^{\emptyset} \geq \Lambda(\lambda,\mu) \geq \Lambda_{\emptyset}^{C} \Rightarrow D\subseteq C}$''}

Assume that $D\not\subseteq C$. Since $D,C\in\mathcal{A}$ and by Assumption~\ref{assump:technical}, there exists an open ball $B\subset D\setminus C$ and a relatively open connected set $U\subset \overline{\Omega}$ with $\partial(\Omega\setminus \overline{U})$ Lipschitz continuous, such that $U$ intersects $\GammaN$, $B\Subset U$, $\overline{U}\cap C = \emptyset$, and one of four options holds:
\begin{alignat*}{2}
	\textup{Case A: }& \overline{U}\cap D \subseteq D_\infty, \quad
	&\textup{Case B: }& \overline{U}\cap D \subseteq \Dp,  \\
	\textup{Case C: }& \overline{U}\cap D \subseteq D_0, \quad &\textup{Case D: }& \overline{U}\cap D \subseteq \Dm.
\end{alignat*}
We consider these cases separately. 

As we shall see below, the combination of these cases proves that $D\not\subseteq C$ implies either $\Lambda_{C}^{\emptyset} \not\geq \Lambda(\lambda,\mu)$ or $\Lambda(\lambda,\mu) \not\geq \Lambda_{\emptyset}^{C}$, which is the contrapositive formulation of the statement we are proving. In particular, cases A and B will contradict the second operator inequality, while cases C and D will contradict the first operator inequality. 

\subsection*{Case A}

Define
\begin{equation*} 
	\lambda_1 = \begin{cases}
		\lambda_0 & \text{in } \Omega\setminus (\Dm\cup\Dp) \\
		\lambda_- & \text{in } \Dm \\
		\lambda_+ & \text{in } \Dp 
	\end{cases}
	\quad\text{and}\quad
	\mu_1 = \begin{cases}
		\mu_0 & \text{in } \Omega\setminus (\Dm\cup\Dp) \\
		\mu_- & \text{in } \Dm \\
		\mu_+ & \text{in } \Dp,
	\end{cases}
\end{equation*}
and
\begin{equation*} 
	\lambda_2 = \begin{cases}
		\lambda_0 & \text{in } \Omega\setminus (\Dm\cup\Dp\cup D_0) \\
		\lambda_- & \text{in } \Dm \\
		\lambda_+ & \text{in } \Dp \\
		0 & \text{in } D_0
	\end{cases}
	\quad\text{and}\quad
	\mu_2 = \begin{cases}
		\mu_0 & \text{in } \Omega\setminus (\Dm\cup\Dp\cup D_0) \\
		\mu_- & \text{in } \Dm \\
		\mu_+ & \text{in } \Dp \\
		0 & \text{in } D_0.
	\end{cases}
\end{equation*}
We shorten the notation and denote $\Lambda = \Lambda(\lambda,\mu)$ and $\Lambda_j = \Lambda(\lambda_j,\mu_j)$ for $j\in\{0,1,2\}$. Now we can use Lemmas~\ref{lemma:locpot}--\ref{lemma:locpotextreme} (cf.~Remark~\ref{remark:loc}) to pick a sequence $(g_i)$ in $L^2(\GammaN)^d$, such that $u_i = u_{g_i}^{\lambda_0,\mu_0}$, $\widehat{u}_i = u_{g_i}^{\lambda_1,\mu_1}$, and $\widetilde{u}_i = Eu_{g_i}^{\lambda_2,\mu_2}$ satisfy
\begin{equation*}
	\lim_{i\to\infty}\int_B \frob{\widehat{\nabla}\widetilde{u}_i}^2\,\di x = \infty,
\end{equation*}
\begin{equation*}
	\lim_{i\to\infty}\int_{\Omega\setminus U} \frob{\widehat{\nabla}u_i}^2\,\di x = \lim_{i\to\infty}\int_{\Omega\setminus U} \frob{\widehat{\nabla}\widehat{u}_i}^2\,\di x = \lim_{i\to\infty}\int_{\Omega\setminus U} \frob{\widehat{\nabla}\widetilde{u}_i}^2\,\di x = 0,
\end{equation*}
and
\begin{equation*}
	\lim_{i\to\infty}\int_{\Omega\setminus U} \abs{\nabla\cdot u_i}^2\,\di x = \lim_{i\to\infty}\int_{\Omega\setminus U} \abs{\nabla\cdot\widehat{u}_i}^2\,\di x = \lim_{i\to\infty}\int_{\Omega\setminus U} \abs{\nabla\cdot\widetilde{u}_i}^2\,\di x = 0.
\end{equation*}
We may write
\begin{equation} \label{eq:caseA}
	\Lambda-\Lambda_{\emptyset}^C = (\Lambda - \Lambda_2) + (\Lambda_2-\Lambda_1) + (\Lambda_1-\Lambda_0) + (\Lambda_0-\Lambda_\emptyset^C),
\end{equation}
and separately estimate each of the above differences. Using Lemmas~\ref{lemma:monoback}--\ref{lemma:monoC0} gives:
\begin{align}
	\inner{(\Lambda - \Lambda_2)g_i,g_i} &\leq  -\int_{D_\infty}2\mu_0\frob{\widehat{\nabla}\widetilde{u}_i}^2\,\di x, \label{eq:caseA1}\\
	\inner{(\Lambda_2 - \Lambda_1)g_i,g_i} &\leq  \int_{D_0}\lambda_0\abs{\nabla\cdot\widetilde{u}_i}^2 + 2\mu_0\frob{\widehat{\nabla}\widetilde{u}_i}^2\,\di x, \label{eq:caseA2}\\
	\inner{(\Lambda_1-\Lambda_0)g_i,g_i} &\leq \int_{\Dm\cup\Dp} (\lambda_0-\lambda_1)\abs{\nabla\cdot\widehat{u}_i}^2 + 2(\mu_0-\mu_1)\frob{\widehat{\nabla}\widehat{u}_i}^2\,\di x, \label{eq:caseA3}\\
	\inner{(\Lambda_0-\Lambda_\emptyset^C)g_i,g_i} &\leq K\int_C \frob{\widehat{\nabla}u_i}^2\,\di x. \label{eq:caseA4}
\end{align}
The right-hand sides of \eqref{eq:caseA2}--\eqref{eq:caseA4} tend to zero as $C\cup D_0\cup\Dm\cup\Dp \subset \Omega\setminus\overline{U}$, while \eqref{eq:caseA1} tends to minus infinity for $i\to\infty$ as $B\subset D_\infty$. From \eqref{eq:caseA} we thus have
\begin{equation*}
	\lim_{i\to\infty}\inner{(\Lambda(\lambda,\mu)-\Lambda_{\emptyset}^C)g_i,g_i} = -\infty,
\end{equation*}
which means that $\Lambda(\lambda,\mu)\not\geq \Lambda_{\emptyset}^C$.

\subsection*{Case B}

Define
\begin{equation*} 
	\lambda_1 = \begin{cases}
		\lambda_0 & \text{in } \Omega\setminus D_0 \\
		0 & \text{in } D_0 
	\end{cases}
	\quad\text{and}\quad
	\mu_1 = \begin{cases}
		\mu_0 & \text{in } \Omega\setminus D_0 \\
		0 & \text{in } D_0,
	\end{cases}
\end{equation*}
and
\begin{equation*} 
	\lambda_2 = \begin{cases}
		\lambda_0 & \text{in } \Omega\setminus (D_0\cup D_\infty) \\
		0 & \text{in } D_0  \\
		\infty & \text{in } D_\infty
	\end{cases}
	\quad\text{and}\quad
	\mu_2 = \begin{cases}
		\mu_0 & \text{in } \Omega\setminus (D_0\cup D_\infty) \\
		0 & \text{in } D_0 \\
		\infty & \text{in } D_\infty.
	\end{cases}
\end{equation*}
We shorten the notation and denote $\Lambda = \Lambda(\lambda,\mu)$ and $\Lambda_j = \Lambda(\lambda_j,\mu_j)$ for $j\in\{0,1,2\}$. Now we can use Lemma~\ref{lemma:locpot} and Lemma~\ref{lemma:locpotextreme} to pick a sequence $(g_i)$ in $L^2(\GammaN)^d$, such that $u_i = u_{g_i}^{\lambda_0,\mu_0}$, $\widehat{u}_i = Eu_{g_i}^{\lambda_1,\mu_1}$, and $\widetilde{u}_i = Eu_{g_i}^{\lambda_2,\mu_2}$ satisfy
\begin{equation*}
	\lim_{i\to\infty}\int_B \frob{\widehat{\nabla}\widetilde{u}_i}^2\,\di x = \infty,
\end{equation*}
\begin{equation*}
	\lim_{i\to\infty}\int_{\Omega\setminus U} \frob{\widehat{\nabla}u_i}^2\,\di x = \lim_{i\to\infty}\int_{\Omega\setminus U} \frob{\widehat{\nabla}\widehat{u}_i}^2\,\di x = \lim_{i\to\infty}\int_{\Omega\setminus U} \frob{\widehat{\nabla}\widetilde{u}_i}^2\,\di x = 0,
\end{equation*}
and
\begin{equation*}
	\lim_{i\to\infty}\int_{\Omega\setminus U} \abs{\nabla\cdot u_i}^2\,\di x = \lim_{i\to\infty}\int_{\Omega\setminus U} \abs{\nabla\cdot\widehat{u}_i}^2\,\di x = \lim_{i\to\infty}\int_{\Omega\setminus U} \abs{\nabla\cdot\widetilde{u}_i}^2\,\di x = 0.
\end{equation*}
We may write
\begin{equation} \label{eq:caseB}
	\Lambda-\Lambda_{\emptyset}^C = (\Lambda - \Lambda_2) + (\Lambda_2-\Lambda_1) + (\Lambda_1-\Lambda_0) + (\Lambda_0-\Lambda_\emptyset^C),
\end{equation}
and separately estimate each of the above of differences. Using Lemmas~\ref{lemma:monoback}--\ref{lemma:monoC0} gives:
\begin{align}
	\inner{(\Lambda - \Lambda_2)g_i,g_i} &\leq  \int_{\Dm\cup\Dp} \frac{\lambda_0}{\lambda}(\lambda_0-\lambda)\abs{\nabla\cdot \widetilde{u}_i}^2 + 2\frac{\mu_0}{\mu}(\mu_0-\mu)\frob{\widehat{\nabla}\widetilde{u}_i}^2\,\di x, \label{eq:caseB1}\\
	\inner{(\Lambda_2 - \Lambda_1)g_i,g_i} &\leq 0, \label{eq:caseB2}\\
	\inner{(\Lambda_1-\Lambda_0)g_i,g_i} &\leq \int_{D_0}\lambda_0\abs{\nabla\cdot\widehat{u}_i}^2+2\mu_0\frob{\widehat{\nabla}\widehat{u}_i}^2\,\di x, \label{eq:caseB3}\\
	\inner{(\Lambda_0-\Lambda_\emptyset^C)g_i,g_i} &\leq K\int_C \frob{\widehat{\nabla}u_i}^2\,\di x. \label{eq:caseB4}
\end{align}
The right-hand sides of \eqref{eq:caseB3}, \eqref{eq:caseB4}, and the part of the integral over $\Dm$ in \eqref{eq:caseB1} tend to zero as $C\cup D_0\cup \Dm \subset \Omega\setminus\overline{U}$. By Assumption~\ref{assump:technical} we may assume that $\inf_B(\mu_+-\mu_0)>0$. Since $B\subset \Dp$, $\lambda_+\geq\lambda_0$, and $\mu_+\geq \mu_0$, the right-hand side of \eqref{eq:caseB1} tends to minus infinity for $i\to\infty$. From \eqref{eq:caseB} we thus have
\begin{equation*}
	\lim_{i\to\infty}\inner{(\Lambda(\lambda,\mu)-\Lambda_{\emptyset}^C)g_i,g_i} = -\infty,
\end{equation*}
which means that $\Lambda(\lambda,\mu)\not\geq \Lambda_{\emptyset}^C$.

\subsection*{Case C}

Define
\begin{equation*} 
	\lambda_1 = \begin{cases}
		\lambda_0 & \text{in } \Omega\setminus (\Dm\cup\Dp) \\
		\lambda_- & \text{in } \Dm \\
		\lambda_+ & \text{in } \Dp 
	\end{cases}
	\quad\text{and}\quad
	\mu_1 = \begin{cases}
		\mu_0 & \text{in } \Omega\setminus (\Dm\cup\Dp) \\
		\mu_- & \text{in } \Dm \\
		\mu_+ & \text{in } \Dp,
	\end{cases}
\end{equation*}
and
\begin{equation*} 
	\lambda_2 = \begin{cases}
		\lambda_0 & \text{in } \Omega\setminus (\Dm\cup\Dp\cup D_\infty) \\
		\lambda_- & \text{in } \Dm \\
		\lambda_+ & \text{in } \Dp \\
		\infty & \text{in } D_\infty
	\end{cases}
	\quad\text{and}\quad
	\mu_2 = \begin{cases}
		\mu_0 & \text{in } \Omega\setminus (\Dm\cup\Dp\cup D_\infty) \\
		\mu_- & \text{in } \Dm \\
		\mu_+ & \text{in } \Dp \\
		\infty & \text{in } D_\infty.
	\end{cases}
\end{equation*}
Moreover, set
\begin{equation*} 
	\lambda_C = \begin{cases}
		\lambda_0 & \text{in } \Omega\setminus C \\
		0 & \text{in } C 
	\end{cases}
	\quad\text{and}\quad
	\mu_C = \begin{cases}
		\mu_0 & \text{in } \Omega\setminus C \\
		0 & \text{in } C.
	\end{cases}
\end{equation*}
We shorten the notation and denote $\Lambda = \Lambda(\lambda,\mu)$ and $\Lambda_j = \Lambda(\lambda_j,\mu_j)$ for $j\in\{0,1,2\}$. Now we can use Lemmas~\ref{lemma:locpot}--\ref{lemma:locpotextreme} (cf.~Remark~\ref{remark:loc}) to pick a sequence $(g_i)$ in $L^2(\GammaN)^d$, such that $u_i = u_{g_i}^{\lambda_0,\mu_0}$, $\widehat{u}_i = u_{g_i}^{\lambda_1,\mu_1}$, $\widetilde{u}_i = u_{g_i}^{\lambda_2,\mu_2}$, and $\breve{u}_i = Eu_{g_i}^{\lambda_C,\mu_C}$ satisfy
\begin{equation*}
	\lim_{i\to\infty}\int_B \frob{\widehat{\nabla}\widetilde{u}_i}^2\,\di x = \infty,
\end{equation*}
\begin{equation*}
	\lim_{i\to\infty}\int_{\Omega\setminus U} \frob{\widehat{\nabla}u_i}^2\,\di x = \lim_{i\to\infty}\int_{\Omega\setminus U} \frob{\widehat{\nabla}\widehat{u}_i}^2\,\di x = \lim_{i\to\infty}\int_{\Omega\setminus U} \frob{\widehat{\nabla}\widetilde{u}_i}^2\,\di x = \lim_{i\to\infty}\int_{\Omega\setminus U} \frob{\widehat{\nabla}\breve{u}_i}^2\,\di x = 0,
\end{equation*}
and
\begin{equation*}
	\lim_{i\to\infty}\int_{\Omega\setminus U} \abs{\nabla\cdot u_i}^2\,\di x = \lim_{i\to\infty}\int_{\Omega\setminus U} \abs{\nabla\cdot\widehat{u}_i}^2\,\di x = \lim_{i\to\infty}\int_{\Omega\setminus U} \abs{\nabla\cdot\widetilde{u}_i}^2\,\di x = \lim_{i\to\infty}\int_{\Omega\setminus U} \abs{\nabla\cdot\breve{u}_i}^2\,\di x = 0.
\end{equation*}
We may write
\begin{equation} \label{eq:caseC}
	\Lambda_C^\emptyset - \Lambda = (\Lambda_C^\emptyset - \Lambda_0) + (\Lambda_0-\Lambda_1) + (\Lambda_1-\Lambda_2) + (\Lambda_2-\Lambda),
\end{equation}
and separately estimate each of the above differences. Using Lemmas~\ref{lemma:monoback}--\ref{lemma:monoC0} gives:
\begin{align}
	\inner{(\Lambda_C^\emptyset - \Lambda_0)g_i,g_i} &\leq  \int_{C} \lambda_0\abs{\nabla\cdot \breve{u}_i}^2 + 2\mu_0\frob{\widehat{\nabla}\breve{u}_i}^2\,\di x, \label{eq:caseC1}\\
	\inner{(\Lambda_0 - \Lambda_1)g_i,g_i} &\leq  \int_{\Dm\cup\Dp} (\lambda_1-\lambda_0)\abs{\nabla\cdot u_i}^2 + 2(\mu_1-\mu_0)\frob{\widehat{\nabla}u_i}^2\,\di x, \label{eq:caseC2}\\
	\inner{(\Lambda_1-\Lambda_2)g_i,g_i} &\leq K\int_{D_\infty} \frob{\widehat{\nabla}\widehat{u}_i}^2\,\di x, \label{eq:caseC3}\\
	\inner{(\Lambda_2-\Lambda)g_i,g_i} &\leq -\int_{D_0} 2\mu_0\frob{\widehat{\nabla}\widetilde{u}_i}^2\,\di x. \label{eq:caseC4}
\end{align}
The right-hand sides of \eqref{eq:caseC1}--\eqref{eq:caseC3} tend to zero as $C\cup\Dm\cup\Dp\cup D_\infty \subset \Omega\setminus\overline{U}$, while \eqref{eq:caseC4} tends to minus infinity for $i\to\infty$ as $B\subset D_0$. From \eqref{eq:caseC} we thus have
\begin{equation*}
	\lim_{i\to\infty}\inner{(\Lambda_{C}^{\emptyset}  - \Lambda(\lambda,\mu))g_i,g_i} = -\infty,
\end{equation*}
which means that $\Lambda_{C}^{\emptyset} \not\geq \Lambda(\lambda,\mu)$.

\subsection*{Case D}

Define
\begin{equation*} 
	\lambda_1 = \begin{cases}
		\lambda_0 & \text{in } \Omega\setminus D_\infty \\
		\infty & \text{in } D_\infty
	\end{cases}
	\quad\text{and}\quad
	\mu_1 = \begin{cases}
		\mu_0 & \text{in } \Omega\setminus D_\infty \\
		\infty & \text{in } D_\infty,
	\end{cases}
\end{equation*}
and
\begin{equation*} 
	\lambda_2 = \begin{cases}
		\lambda_0 & \text{in } \Omega\setminus (D_0\cup D_\infty) \\
		0 & \text{in } D_0 \\
		\infty & \text{in } D_\infty
	\end{cases}
	\quad\text{and}\quad
	\mu_2 = \begin{cases}
		\mu_0 & \text{in } \Omega\setminus (D_0\cup D_\infty) \\
		0 & \text{in } D_0 \\
		\infty & \text{in } D_\infty.
	\end{cases}
\end{equation*}
Moreover, set
\begin{equation*} 
	\lambda_C = \begin{cases}
		\lambda_0 & \text{in } \Omega\setminus C \\
		0 & \text{in } C 
	\end{cases}
	\quad\text{and}\quad
	\mu_C = \begin{cases}
		\mu_0 & \text{in } \Omega\setminus C \\
		0 & \text{in } C.
	\end{cases}
\end{equation*}
We shorten the notation and denote $\Lambda = \Lambda(\lambda,\mu)$ and $\Lambda_j = \Lambda(\lambda_j,\mu_j)$ for $j\in\{0,1,2\}$. Now we can use Lemma~\ref{lemma:locpot} and Lemma~\ref{lemma:locpotextreme} to pick a sequence $(g_i)$ in $L^2(\GammaN)^d$, such that $u_i = u_{g_i}^{\lambda_0,\mu_0}$, $\widetilde{u}_i = Eu_{g_i}^{\lambda_2,\mu_2}$, and $\breve{u}_i = Eu_{g_i}^{\lambda_C,\mu_C}$ satisfy
\begin{equation*}
	\lim_{i\to\infty}\int_B \frob{\widehat{\nabla}\widetilde{u}_i}^2\,\di x = \infty,
\end{equation*}
\begin{equation*}
	\lim_{i\to\infty}\int_{\Omega\setminus U} \frob{\widehat{\nabla}u_i}^2\,\di x = \lim_{i\to\infty}\int_{\Omega\setminus U} \frob{\widehat{\nabla}\widetilde{u}_i}^2\,\di x = \lim_{i\to\infty}\int_{\Omega\setminus U} \frob{\widehat{\nabla}\breve{u}_i}^2\,\di x = 0,
\end{equation*}
and
\begin{equation*}
	\lim_{i\to\infty}\int_{\Omega\setminus U} \abs{\nabla\cdot u_i}^2\,\di x = \lim_{i\to\infty}\int_{\Omega\setminus U} \abs{\nabla\cdot\widetilde{u}_i}^2\,\di x = \lim_{i\to\infty}\int_{\Omega\setminus U} \abs{\nabla\cdot\breve{u}_i}^2\,\di x = 0.
\end{equation*}
We may write
\begin{equation} \label{eq:caseD}
	\Lambda_C^\emptyset - \Lambda = (\Lambda_C^\emptyset - \Lambda_0) + (\Lambda_0-\Lambda_1) + (\Lambda_1-\Lambda_2) + (\Lambda_2-\Lambda),
\end{equation}
and separately estimate each of the above differences. Using Lemmas~\ref{lemma:monoback}--\ref{lemma:monoC0} gives:
\begin{align}
	\inner{(\Lambda_C^\emptyset - \Lambda_0)g_i,g_i} &\leq  \int_{C} \lambda_0\abs{\nabla\cdot \breve{u}_i}^2 + 2\mu_0\frob{\widehat{\nabla}\breve{u}_i}^2\,\di x, \label{eq:caseD1}\\
	\inner{(\Lambda_0 - \Lambda_1)g_i,g_i} &\leq  K\int_{D_\infty} \frob{\widehat{\nabla}u_i}^2\,\di x, \label{eq:caseD2}\\
	\inner{(\Lambda_1-\Lambda_2)g_i,g_i} &\leq 0, \label{eq:caseD3}\\
	\inner{(\Lambda_2-\Lambda)g_i,g_i} &\leq \int_{\Dm\cup\Dp} (\lambda-\lambda_0)\abs{\nabla\cdot \widetilde{u}_i}^2 + 2(\mu-\mu_0)\frob{\widehat{\nabla}\widetilde{u}_i}^2\,\di x. \label{eq:caseD4}
\end{align}
The right-hand sides of \eqref{eq:caseD1}, \eqref{eq:caseD2}, and the part of the integral over $\Dp$ in \eqref{eq:caseD4} tend to zero as $C\cup D_\infty\cup\Dp \subset \Omega\setminus\overline{U}$. By Assumption~\ref{assump:technical} we may assume that $\sup_B(\mu_- - \mu_0)<0$. Since $B\subset \Dm$, $\lambda_- \leq\lambda_0$, and $\mu_- \leq \mu_0$, the right-hand side of \eqref{eq:caseD4} tends to minus infinity for $i\to\infty$. From \eqref{eq:caseD} we thus have
\begin{equation*}
	\lim_{i\to\infty}\inner{(\Lambda_{C}^{\emptyset}  - \Lambda(\lambda,\mu))g_i,g_i} = -\infty,
\end{equation*}
which means that $\Lambda_{C}^{\emptyset} \not\geq \Lambda(\lambda,\mu)$.

\section{Proof of Theorem~\ref{thm:monopos}} \label{sec:monoposproof}

\subsection*{Proof of ``$\boldsymbol{B\subset D \Rightarrow \Lambda_{\beta,B}\geq \Lambda_\emptyset^{D}}$''}

Define the following truncated Lam\'e parameters:
\begin{equation} \label{eq:epspos}
	\lambda_\epsilon = \begin{cases}
		\lambda_0 & \text{in } \Omega\setminus D\\
		\epsilon^{-1}\lambda_0 & \text{in } D
	\end{cases}
	\quad \text{and}
	\quad
	\mu_\epsilon = \begin{cases}
		\mu_0 & \text{in } \Omega\setminus D\\
		\epsilon^{-1}\mu_0 & \text{in } D.
	\end{cases}
\end{equation}
As $B\subset D$, for small enough $\epsilon>0$ we have $\lambda_0+\beta\chi_B \leq \lambda_\epsilon$ and $\mu_0+\beta\chi_B \leq \mu_\epsilon$. Thus Lemma~\ref{lemma:monoback} implies $\Lambda_{\beta,B}\geq \Lambda(\lambda_\epsilon,\mu_\epsilon)$. Now the convergence in operator norm as $\epsilon\to 0$ by Theorem~\ref{thm:convergence} implies $\Lambda_{\beta,B}\geq \Lambda_\emptyset^{D}$.

\subsection*{Proof of ``$\boldsymbol{B\subset D \Rightarrow \Lambda_0 + \DLambda_{\beta,B}\geq \Lambda_\emptyset^{D}}$''}

We use the notation
\begin{equation} \label{eq:kappa}
	\kappa = \min\{\inf(\lambda_0),\inf(\mu_0)\},
\end{equation}
and fix an $\alpha_0>0$ satisfying
\begin{equation} \label{eq:alpha0}
	\beta \leq \frac{\alpha_0}{1+\alpha_0}\kappa.
\end{equation}
Now we consider a small enough $\epsilon>0$ such that
\begin{equation} \label{eq:epsalph1}
	1-\epsilon \geq \frac{\alpha_0}{1+\alpha_0}.
\end{equation}
Let $\lambda_\epsilon$ and $\mu_\epsilon$ be as in \eqref{eq:epspos} and $u_0 = u_g^{\lambda_0,\mu_0}$. The set inclusion $B\subset D$ combined with Lemma~\ref{lemma:monoback} and \eqref{eq:kappa}--\eqref{eq:epsalph1} gives
\begin{align*}
	\inner{(\Lambda_0 + \DLambda_{\beta,B} - \Lambda(\lambda_\epsilon,\mu_\epsilon))g,g} \hspace{-3cm}& \\
	&\geq \int_D (1-\epsilon)\lambda_0\abs{\nabla u_0}^2 + 2(1-\epsilon)\mu_0\frob{\widehat{\nabla}u_0}^2\,\di x -\beta \int_B \abs{\nabla\cdot u_0}^2 + 2\frob{\widehat{\nabla}u_0}^2\, \di x \\
	&\geq \int_B [(1-\epsilon)\lambda_0-\beta]\abs{\nabla u_0}^2 + 2[(1-\epsilon)\mu_0-\beta]\frob{\widehat{\nabla}u_0}^2\,\di x \\
	&\geq \frac{\alpha_0}{1+\alpha_0}\int_B (\lambda_0-\kappa)\abs{\nabla u_0}^2 + 2(\mu_0-\kappa)\frob{\widehat{\nabla}u_0}^2\,\di x \geq 0.
\end{align*}
Using Theorem~\ref{thm:convergence}, the operator norm convergence for $\epsilon\to 0$ yields $\Lambda_0 + \DLambda_{\beta,B} \geq \Lambda_\emptyset^{D}$.

\subsection*{Proof of ``$\boldsymbol{\Lambda_{\beta,B}\geq \Lambda_\emptyset^{D} \Rightarrow B\subset D}$''}

Let $u_0 = u_g^{\lambda_0,\mu_0}$. Lemma~\ref{lemma:monoback} and Lemma~\ref{lemma:monoCinfty} give
\begin{align}
	\inner{(\Lambda_{\beta,B}-\Lambda_\emptyset^D)g,g} &= \inner{(\Lambda_{\beta,B}-\Lambda_0)g,g} + \inner{(\Lambda_0-\Lambda_\emptyset^D)g,g} \notag\\
	&\leq K\int_D \frob{\widehat{\nabla}u_0}^2\,\di x - \beta\int_B \frac{\lambda_0}{\lambda_0+\beta}\abs{\nabla\cdot u_0}^2 + 2\frac{\mu_0}{\mu_0+\beta}\frob{\widehat{\nabla}u_0}^2\,\di x. \label{eq:extramonopos}
\end{align}
Assume $B\not\subset D$ so that $B\setminus D$ contains an open ball $\widehat{B}$, in which we may concentrate localized potentials via a suitable set $U$. To this end, let $U\subset\overline{\Omega}$ be a relatively open connected set that intersects $\GammaN$, compactly contains $\widehat{B}$, and satisfies $\overline{U}\cap D = \emptyset$. Pick $(g_i)$ in $L^2_\diamond(\GammaN)^d$ via Lemma~\ref{lemma:locpot} such that $u_i = u_{g_i}^{\lambda_0,\mu_0}$ satisfies
\begin{equation*}
	\lim_{i\to\infty}\int_B \abs{\nabla\cdot u_i}^2\,\di x = \lim_{i\to\infty}\int_B \frob{\widehat{\nabla}u_i}^2\,\di x = \infty 
\end{equation*}
and
\begin{equation*}
	\lim_{i\to\infty}\int_D \abs{\nabla\cdot u_i}^2\,\di x = \lim_{i\to\infty}\int_D \frob{\widehat{\nabla}u_i}^2\,\di x = 0,
\end{equation*}
where we used the set inclusions $D\subset\Omega\setminus\overline{U}$ and $\widehat{B}\subseteq B$. Now \eqref{eq:extramonopos} gives
\begin{align*}
	\inner{(\Lambda_{\beta,B}-\Lambda_\emptyset^D)g_i,g_i} \hspace{-2cm}& \\
	&\leq K\int_D \frob{\widehat{\nabla}u_i}^2\,\di x - \beta\int_B \frac{\lambda_0}{\lambda_0+\beta}\abs{\nabla\cdot u_i}^2 + 2\frac{\mu_0}{\mu_0+\beta}\frob{\widehat{\nabla}u_i}^2\,\di x \to -\infty
\end{align*}
for $i\to\infty$. Thereby $B\not\subset D$ implies $\Lambda_{\beta,B} \not\geq \Lambda_\emptyset^{D}$, which is the contrapositive formulation of the statement we are proving.

\subsection*{Proof of ``$\boldsymbol{\Lambda_0 + \DLambda_{\beta,B}\geq \Lambda_\emptyset^{D} \Rightarrow B\subset D}$''}

Assume $B\not\subset D$. We pick $(g_i)$ in $L^2_\diamond(\GammaN)^d$ and $u_i = u_{g_i}^{\lambda_0,\mu_0}$ as in the proof of ``$\Lambda_{\beta,B}\geq \Lambda_\emptyset^{D} \Rightarrow B\subset D$''. Now Lemma~\ref{lemma:monoCinfty} gives
\begin{align*}
	\inner{(\Lambda_0 + \DLambda_{\beta,B} - \Lambda_\emptyset^{D})g_i,g_i} \hspace{-2cm}& \\
	&\leq K\int_D \frob{\widehat{\nabla}u_i}^2\,\di x -\beta \int_B \abs{\nabla\cdot u_i}^2 + 2\frob{\widehat{\nabla}u_i}^2\, \di x \to -\infty
\end{align*}
for $i\to\infty$. Thereby $B\not\subset D$ implies $\Lambda_0 + \DLambda_{\beta,B} \not\geq \Lambda_\emptyset^{D}$, which is the contrapositive formulation of the statement we are proving.

\section{Proof of Theorem~\ref{thm:mononeg}} \label{sec:mononegproof}

\subsection*{Proof of ``$\boldsymbol{B\subset D \Rightarrow \Lambda_D^{\emptyset} \geq \Lambda_{-\beta,B}}$''}

Define the following truncated Lam\'e parameters:
\begin{equation}  \label{eq:epsneg}
	\lambda_\epsilon = \begin{cases}
		\lambda_0 & \text{in } \Omega\setminus D\\
		\epsilon\lambda_0 & \text{in } D
	\end{cases}
	\quad \text{and}
	\quad
	\mu_\epsilon = \begin{cases}
		\mu_0 & \text{in } \Omega\setminus D\\
		\epsilon\mu_0 & \text{in } D.
	\end{cases}
\end{equation}
As $B\subset D$ and $\beta < \kappa$ by \eqref{eq:kappa}, for small enough $\epsilon>0$ we have $\lambda_\epsilon \leq \lambda_0-\beta\chi_B$ and $\mu_\epsilon \leq \mu_0-\beta\chi_B$. Thus Lemma~\ref{lemma:monoback} implies $\Lambda(\lambda_\epsilon,\mu_\epsilon) \geq \Lambda_{-\beta,B}$. Now the convergence in operator norm as $\epsilon\to 0$ guaranteed by Theorem~\ref{thm:convergence} implies $\Lambda_D^{\emptyset} \geq \Lambda_{-\beta,B}$.

\subsection*{Proof of ``$\boldsymbol{B\subset D \Rightarrow \Lambda_D^{\emptyset} \geq \Lambda_0 + \DLambda_{-\beta,B}}$''}

Let $\lambda_\epsilon$ and $\mu_\epsilon$ be as in \eqref{eq:epsneg} and $u_0 = u_g^{\lambda_0,\mu_0}$. The set inclusion $B\subset D$ combined with Lemma~\ref{lemma:monoback} and \eqref{eq:kappa}--\eqref{eq:epsalph1} gives
\begin{align*}
	\inner{(\Lambda(\lambda_\epsilon,\mu_\epsilon) - \Lambda_0 - \DLambda_{-\beta,B})g,g} \hspace{-3cm}& \\
	&\geq \int_D (1-\epsilon)\lambda_0\abs{\nabla u_0}^2 + 2(1-\epsilon)\mu_0\frob{\widehat{\nabla}u_0}^2\,\di x -\beta \int_B \abs{\nabla\cdot u_0}^2 + 2\frob{\widehat{\nabla}u_0}^2\, \di x \\
	&\geq \int_B [(1-\epsilon)\lambda_0-\beta]\abs{\nabla u_0}^2 + 2[(1-\epsilon)\mu_0-\beta]\frob{\widehat{\nabla}u_0}^2\,\di x \\
	&\geq \frac{\alpha_0}{1+\alpha_0}\int_B (\lambda_0-\kappa)\abs{\nabla u_0}^2 + 2(\mu_0-\kappa)\frob{\widehat{\nabla}u_0}^2\,\di x \geq 0.
\end{align*}
Using Theorem~\ref{thm:convergence}, the operator norm convergence for $\epsilon\to 0$ yields $\Lambda_D^{\emptyset} \geq \Lambda_0 + \DLambda_{-\beta,B}$.

\subsection*{Proof of ``$\boldsymbol{\Lambda_D^{\emptyset} \geq \Lambda_{-\beta,B} \Rightarrow B\subset D}$''}

We set
\begin{equation}  \label{eq:Dneg}
	\lambda = \begin{cases}
		\lambda_0 & \text{in } \Omega\setminus D\\
		0 & \text{in } D
	\end{cases}
	\quad \text{and}
	\quad
	\mu = \begin{cases}
		\mu_0 & \text{in } \Omega\setminus D\\
		0 & \text{in } D.
	\end{cases}
\end{equation}
Let $\lambda_\epsilon$ and $\mu_\epsilon$ be as in \eqref{eq:epsneg} and $u_\epsilon = u_g^{\lambda_\epsilon,\mu_\epsilon}$. Then Lemma~\ref{lemma:monoback} gives
\begin{equation*}
	\inner{(\Lambda(\lambda_\epsilon,\mu_\epsilon)-\Lambda_{-\beta,B})g,g} \leq (1-\epsilon)\int_D\lambda_0\abs{\nabla\cdot u_\epsilon}^2+2\mu_0\frob{\widehat{\nabla}u_\epsilon}^2\,\di x - \beta\int_B\abs{\nabla\cdot u_\epsilon}^2+2\frob{\widehat{\nabla}u_\epsilon}^2\,\di x.
\end{equation*}
Letting $\epsilon\to 0$ and using Theorem~\ref{thm:convergence} yields
\begin{equation} \label{eq:extramononeg}
	\inner{(\Lambda_D^\emptyset-\Lambda_{-\beta,B})g,g} \leq \int_D\lambda_0\abs{\nabla\cdot \widehat{u}}^2+2\mu_0\frob{\widehat{\nabla}\widehat{u}}^2\,\di x - \beta\int_B\abs{\nabla\cdot \widehat{u}}^2+2\frob{\widehat{\nabla}\widehat{u}}^2\,\di x,
\end{equation}
where $\widehat{u} = Eu_g^{\lambda,\mu}$.

Assume $B\not\subset D$ so that $B\setminus D$ contains an open ball $\widehat{B}$, in which we may concentrate localized potentials via a suitable set $U$. To this end, let $U\subset\overline{\Omega}$ be a relatively open connected set that intersects $\GammaN$, compactly contains $\widehat{B}$, and satisfies $\overline{U}\cap D = \emptyset$. We now pick $(g_i)$ in $L^2_\diamond(\GammaN)^d$ via Lemma~\ref{lemma:locpot} and Lemma~\ref{lemma:locpotextreme} such that $\widehat{u}_i = Eu_{g_i}^{\lambda,\mu}$ satisfies
\begin{equation*}
	\lim_{i\to\infty}\int_B \abs{\nabla\cdot \widehat{u}_i}^2\,\di x = \lim_{i\to\infty}\int_B \frob{\widehat{\nabla}\widehat{u}_i}^2\,\di x = \infty
\end{equation*}
and
\begin{equation*}
	\lim_{i\to\infty}\int_D \abs{\nabla\cdot \widehat{u}_i}^2\,\di x = \lim_{i\to\infty}\int_D \frob{\widehat{\nabla}\widehat{u}_i}^2\,\di x = 0,
\end{equation*}
where we used the set inclusions $D\subset\Omega\setminus\overline{U}$ and $\widehat{B}\subseteq B$. Now \eqref{eq:extramononeg} gives
\begin{align*}
	\inner{(\Lambda_D^{\emptyset} - \Lambda_{-\beta,B})g_i,g_i} \hspace{-2cm}& \\
	&\leq \int_D \lambda_0\abs{\nabla\cdot \widehat{u}_i}^2 + 2\mu_0\frob{\widehat{\nabla}\widehat{u}_i}^2\,\di x -\beta \int_B \abs{\nabla\cdot \widehat{u}_i}^2 + 2\frob{\widehat{\nabla}\widehat{u}_i}^2\, \di x \to -\infty
\end{align*}
for $i\to\infty$. Thereby $B\not\subset D$ implies $\Lambda_D^{\emptyset} \not\geq \Lambda_{-\beta,B}$, which is the contrapositive formulation of the statement we are proving.

\subsection*{Proof of ``$\boldsymbol{\Lambda_D^{\emptyset} \geq \Lambda_0 + \DLambda_{-\beta,B} \Rightarrow B\subset D}$''}

Assume $B\not\subset D$ and let $\lambda$ and $\mu$ be as in \eqref{eq:Dneg}. We pick $(g_i)$ in $L^2_\diamond(\GammaN)^d$ as in the proof of ``$\Lambda_D^{\emptyset} \geq \Lambda_{-\beta,B} \Rightarrow B\subset D$'', such that $u_i = u_{g_i}^{\lambda_0,\mu_0}$ and $\widehat{u}_i = Eu_{g_i}^{\lambda,\mu}$ satisfy
\begin{equation*}
	\lim_{i\to\infty}\int_B \abs{\nabla\cdot u_i}^2\,\di x = \lim_{i\to\infty}\int_B \frob{\widehat{\nabla}u_i}^2\,\di x = \infty
\end{equation*}
and
\begin{equation*}
	\lim_{i\to\infty}\int_D \abs{\nabla\cdot \widehat{u}_i}^2\,\di x = \lim_{i\to\infty}\int_D \frob{\widehat{\nabla}\widehat{u}_i}^2\,\di x = 0.
\end{equation*}
Now Lemma~\ref{lemma:monoC0} gives
\begin{align*}
	\inner{(\Lambda_D^\emptyset - \Lambda_0 - \DLambda_{-\beta,B})g_i,g_i} \hspace{-2cm}& \\
	&\leq \int_D \lambda_0\abs{\nabla\cdot \widehat{u}_i}^2 + 2\mu_0\frob{\widehat{\nabla}\widehat{u}_i}^2\,\di x -\beta \int_B \abs{\nabla\cdot u_i}^2 + 2\frob{\widehat{\nabla}u_i}^2\, \di x \to -\infty
\end{align*}
for $i\to\infty$. Thereby $B\not\subset D$ implies $\Lambda_D^{\emptyset} \not\geq \Lambda_0 + \DLambda_{-\beta,B}$, which is the contrapositive formulation of the statement we are proving.

\section{Proof of Theorem~\ref{thm:monofinitelin}} \label{sec:monofinitelinptoof}

Recall from Definition~\ref{def:posneginc} that $D = \Dm\cup\Dp$ in case of \emph{non-extreme} inclusions, and $\lambda$ and $\mu$ are on the form
\begin{equation*} 
	\lambda = \begin{cases}
		\lambda_0 & \text{in } \Omega\setminus D \\
		\lambda_- & \text{in } \Dm \\
		\lambda_+ & \text{in } \Dp 
	\end{cases}
	\quad\text{and}\quad
	\mu = \begin{cases}
		\mu_0 & \text{in } \Omega\setminus D \\
		\mu_- & \text{in } \Dm \\
		\mu_+ & \text{in } \Dp,
	\end{cases}
\end{equation*}
where $\alpha_\lambda \leq \lambda_- \leq \lambda_0 \leq \lambda_+ \leq \beta_\lambda$ and $\alpha_\mu\leq \mu_-\leq\mu_0\leq\mu_+\leq \beta_\mu$, and satisfy Assumption~\ref{assump:technical}.

\subsection*{Proof of ``$\boldsymbol{D\subseteq C \Rightarrow \DLambda_C^- \geq \Lambda(\lambda,\mu)-\Lambda(\lambda_0,\mu_0) \geq \DLambda_C^+}$''}

Let $u_0 = u_g^{\lambda_0,\mu_0}$. Since $D\subseteq C$, $\lambda_- \leq \lambda_0$, and $\mu_- \leq \mu_0$, Lemma~\ref{lemma:monoback} gives
\begin{align*}
	\inner{(\Lambda(\lambda,\mu) - \Lambda(\lambda_0,\mu_0) - \DLambda_C^+)g,g} &\geq \int_{\Dp} (\lambda_0 - \lambda_+)\abs{\nabla\cdot u_0}^2 + 2(\mu_0-\mu_+)\frob{\widehat{\nabla}u_0}^2\,\di x \\
	&\hphantom{\geq{}} + \int_C (\beta_\lambda-\lambda_0)\abs{\nabla\cdot u_0}^2 +2(\beta_\mu-\mu_0)\frob{\widehat{\nabla}u_0}^2\,\di x \\
	&\geq \int_{C\setminus \Dp} (\beta_\lambda-\lambda_0)\abs{\nabla\cdot u_0}^2 +2(\beta_\mu-\mu_0)\frob{\widehat{\nabla}u_0}^2\,\di x \\
	&\geq 0.
\end{align*}
Next we consider the other inequality. Note that
\begin{equation*}
	\frac{\lambda_0}{\lambda_-}(\lambda_- - \lambda_0) \geq \lambda_0-\frac{\beta_\lambda^2}{\alpha_\lambda} \quad \text{and} \quad \frac{\mu_0}{\mu_-}(\mu_- - \mu_0) \geq \mu_0-\frac{\beta_\mu^2}{\alpha_\mu}.
\end{equation*}
Since $D\subseteq C$, $\lambda_+\geq\lambda_0$, and $\mu_+\geq\mu_0$, Lemma~\ref{lemma:monoback} gives
\begin{align*}
	\inner{(\Lambda(\lambda_0,\mu_0) - \Lambda(\lambda,\mu) + \DLambda_C^-)g,g} &\geq \int_{\Dm} \frac{\lambda_0}{\lambda_-}(\lambda_- - \lambda_0)\abs{\nabla\cdot u_0}^2 + 2\frac{\mu_0}{\mu_-}(\mu_- - \mu_0)\frob{\widehat{\nabla}u_0}^2\,\di x \\
	&\hphantom{\geq{}} + \int_C (\tfrac{\beta_\lambda^2}{\alpha_\lambda}-\lambda_0)\abs{\nabla\cdot u_0}^2 + 2(\tfrac{\beta_\mu^2}{\alpha_\mu}-\mu_0)\frob{\widehat{\nabla}u_0}^2\,\di x \\
	&\geq \int_{C\setminus \Dm} (\tfrac{\beta_\lambda^2}{\alpha_\lambda}-\lambda_0)\abs{\nabla\cdot u_0}^2 + 2(\tfrac{\beta_\mu^2}{\alpha_\mu}-\mu_0)\frob{\widehat{\nabla}u_0}^2\,\di x \\
	&\geq 0.
\end{align*}

\subsection*{Proof of ``$\boldsymbol{\DLambda_C^- \geq \Lambda(\lambda,\mu)-\Lambda(\lambda_0,\mu_0) \geq \DLambda_C^+ \Rightarrow D\subseteq C}$''}

Assume that $D\not\subseteq C$. Since $D,C\in\widehat{\mathcal{A}}$ and by Assumption~\ref{assump:technical}, there exists an open ball $B\subset D\setminus C$ and a relatively open connected set $U\subset \overline{\Omega}$, such that $U$ intersects $\GammaN$, $B\Subset U$, $\overline{U}\cap C = \emptyset$, and one of two options holds:
\begin{equation*}
	\textup{Case A: } \overline{U}\cap D \subseteq \Dp \quad \text{or} \quad
	\textup{Case B: } \overline{U}\cap D \subseteq \Dm.
\end{equation*}
We consider these cases separately. 

As we shall see below, the combination of the two cases proves that $D\not\subseteq C$ implies either $\DLambda_C^- \not\geq \Lambda(\lambda,\mu)-\Lambda(\lambda_0,\mu_0)$ or $\Lambda(\lambda,\mu)-\Lambda(\lambda_0,\mu_0) \not\geq \DLambda_C^+$, which is the contrapositive formulation of the statement we are proving. In particular, case A will contradict the second operator inequality, while case B will contradict the first operator inequality. 

For both cases we pick $(g_i)$ in $L^2_\diamond(\GammaN)^d$ via Lemma~\ref{lemma:locpot} such that $u_i = u_{g_i}^{\lambda_0,\mu_0}$ satisfies
\begin{equation*}
	\lim_{i\to\infty}\int_B \abs{\nabla\cdot u_i}^2\,\di x = \lim_{i\to\infty}\int_B \frob{\widehat{\nabla}u_i}^2\,\di x = \infty
\end{equation*}
and
\begin{equation*}
	\lim_{i\to\infty}\int_{\Omega\setminus U} \abs{\nabla\cdot u_i}^2\,\di x = \lim_{i\to\infty}\int_{\Omega\setminus U} \frob{\widehat{\nabla}u_i}^2\,\di x = 0.
\end{equation*}

\subsection*{Case A}

By Lemma~\ref{lemma:monoback} we get
\begin{align*}
	\inner{(\Lambda(\lambda,\mu) - \Lambda(\lambda_0,\mu_0) - \DLambda_C^+)g_i,g_i} &\leq \int_{\Dm\cup\Dp} \frac{\lambda_0}{\lambda}(\lambda_0-\lambda)\abs{\nabla\cdot u_i}^2 +  2\frac{\mu_0}{\mu}(\mu_0-\mu)\frob{\widehat{\nabla}u_i}^2\,\di x \\
	&\hphantom{\leq{}} + \int_C (\beta_\lambda-\lambda_0)\abs{\nabla\cdot u_i}^2 +2(\beta_\mu-\mu_0)\frob{\widehat{\nabla}u_i}^2\,\di x. 
\end{align*}
The integrals over $C$ and over $\Dm$ tend to zero as $C\cup\Dm\subset\Omega\setminus\overline{U}$. By Assumption~\ref{assump:technical} we may assume that $\inf_B(\mu_+-\mu_0)>0$. As $B\subset \Dp$, $\lambda_+\geq\lambda_0$, and $\mu_+\geq \mu_0$, the integral over $\Dp$ tends to minus infinity for $i\to\infty$. Thus we conclude that
\begin{equation*}
	\lim_{i\to\infty}\inner{(\Lambda(\lambda,\mu) - \Lambda(\lambda_0,\mu_0) - \DLambda_C^+)g_i,g_i} = -\infty,
\end{equation*}
which means that $\Lambda(\lambda,\mu)-\Lambda(\lambda_0,\mu_0) \not\geq \DLambda_C^+$.

\subsection*{Case B}

By Lemma~\ref{lemma:monoback} we get
\begin{align*}
	\inner{(\Lambda(\lambda_0,\mu_0) - \Lambda(\lambda,\mu) + \DLambda_C^-)g_i,g_i} &\leq \int_{\Dm\cup\Dp} (\lambda-\lambda_0)\abs{\nabla\cdot u_i}^2 + 2(\mu - \mu_0)\frob{\widehat{\nabla}u_i}^2\,\di x \\
	&\hphantom{\leq{}} + \int_C (\tfrac{\beta_\lambda^2}{\alpha_\lambda}-\lambda_0)\abs{\nabla\cdot u_i}^2 + 2(\tfrac{\beta_\mu^2}{\alpha_\mu}-\mu_0)\frob{\widehat{\nabla}u_i}^2\,\di x.
\end{align*}
The integrals over $C$ and over $\Dp$ tend to zero as $C\cup\Dp\subset\Omega\setminus\overline{U}$. By Assumption~\ref{assump:technical} we may assume that $\sup_B(\mu_- - \mu_0)<0$. As $B\subset \Dm$, $\lambda_-\leq\lambda_0$, and $\mu_-\leq \mu_0$, the integral over $\Dm$ tends to minus infinity for $i\to\infty$. Thus we conclude that
\begin{equation*}
	\lim_{i\to\infty}\inner{(\Lambda(\lambda_0,\mu_0) - \Lambda(\lambda,\mu) + \DLambda_C^-)g_i,g_i} = -\infty,
\end{equation*}
which means that $\DLambda_C^- \not\geq \Lambda(\lambda,\mu)-\Lambda(\lambda_0,\mu_0)$.

\subsection*{Acknowledgements}

We thank Valentina Candiani for her participation in early discussions and her contributions to an initial draft of the paper.

HG is supported by grant 10.46540/3120-00003B from Independent Research Fund Denmark. NH is supported by the Research Council of Finland (decisions 353081 and 359181). While at Goethe University Frankfurt, SEB was supported by grant 499303971 from the German Research Foundation.

\bibliographystyle{plain}

\end{document}